\documentclass[12pt,a4paper]{amsart}
\usepackage[a4paper, footskip=0.5in,  headheight = 0.5in, top=1.25in, bottom=1.25in,  right=1in,  left=1in]{geometry}
\usepackage{macros}

\title[Induced packing treewidth II.]{Induced packing treewidth II.\\ Excluding a clique or a biclique}

\author{Amir Nikabadi$^{\dagger}$}\thanks{$^{\dagger}$IT University of Copenhagen, Denmark (\texttt{amir@itu.dk}). Supported by the Independent Research Fund Denmark (DFF), grant agreement number 2098-00012B}
\author{Paweł Rzążewski$^{\parallel}$}\thanks{$^{\parallel}$Warsaw University of Technology, Poland (\texttt{pawel.rzazewski@pw.edu.pl}). Supported by the National Science Centre grant 2024/54/E/ST6/00094.}

\date{}

\begin{document}
\begin{abstract}
The notion of induced packing treewidth aims to unify classes defined by forbidden induced subgraphs or induced minors with classes defined by the existence of certain structured tree decompositions.
For a graph $H$, \emph{induced $H$-packing treewidth}, denoted by $\treepi_{H}$, is
a tree-decomposition-based graph parameter that, for each bag, measures the maximum number of pairwise anticomplete induced copies of $H$ intersecting that bag.
This notion generalizes some previously studied parameters: when $H=P_1$, it is equivalent to tree-independence number, and when $H=P_2$, it is equivalent to induced matching treewidth.

In this paper, we explore the following two questions: (1) Which structural properties of $H$-free graphs
    persist when we only know that $\treepi_{H}$ is bounded? and (2) How does induced packing treewidth relate to
    other graph-width parameters?

We prove the following:
\medskip

\begin{itemize}[itemsep=2mm,leftmargin=6mm]
   \item For all $a,t\in \mathbb{N}$, $K_{a,a}$-free graphs of bounded induced $P_t$-packing treewidth have bounded tree-independence number.
   This extends the previous result of Abrishami et al. [SIAM J. Discrete Math., 2025] for $t=2$, and a result of Hajebi and Spirkl who showed that $(P_t,K_{a,a})$-free graphs have bounded tree-independence number.

   \item If $H$ is any fixed path or a star, then the class of graphs of bounded induced $H$-packing treewidth is $\chi$-bounded.
   Again, this extends the previous result of Abrishami et al. [SIAM J. Discrete Math., 2025] for $H=P_2$.

    \item Finally, we study the relationship between induced packing treewidth and \emph{sim-width}, a width parameter based on branch decompositions.
    We show that, although \emph{sim-width} and induced $P_3$-packing treewidth are incomparable, graphs of bounded sim-width that exclude all \emph{$H$-obstructions}---certain graphs that force large induced $H$-packing treewidth---have bounded induced $H$-packing treewidth.
    This simultaneously generalizes and resolves questions posed by Abrishami et al. [SIAM J. Discrete Math., 2025] and Brettell et al. [European J. Comb., 2025].
\end{itemize}

\end{abstract}

\maketitle

\thispagestyle{empty}
{
 \hypersetup{linkcolor=red!50!black}
 \tableofcontents
}

\section{Introduction}
\subsection{Motivation}
Graph decompositions and their associated width parameters are central tools
for understanding the structure of graphs and for designing algorithms on
them.  The canonical success story is that of tree decompositions and
treewidth, discovered independently, in equivalent forms, by Bertel\`e and
Brioschi~\cite{bertele1972nonserial}, Halin~\cite{halin1976sfunctions}, and
Robertson and Seymour~\cite{robertson1984graph}.  An equivalent formulation in
terms of partial $k$-trees was studied by Arnborg and
Proskurowski~\cite{arnborg1989linear}.
Intuitively, a tree decomposition represents a graph by overlapping sets of
vertices, called \emph{bags}, arranged along a tree: every edge is contained in
a bag, and the bags containing any fixed vertex form a connected subtree.
Treewidth measures how small the bags can be made.  When they have bounded
size, the graph is sufficiently tree-like that many otherwise difficult
problems can be solved by dynamic
programming~\cite{robertson1986graph,bodlaender1988dynamic,
cygan2015parameterized}.

Bounded treewidth, however, is inherently related to sparse graphs:
in every class of bounded treewidth, every graph has only a linear number of edges.
To obtain useful analogues for dense graphs, one
can allow bags to be large while requiring the graphs they induce to have
additional structure.  This viewpoint led Yolov and, independently, Dallard,
Milani\v{c}, and \v{S}torgel to the \emph{tree-independence number}, denoted by
$\treealpha(G)$~\cite{yolov2018minor,dallard2024treewidthII}.  It is defined
just like treewidth, except that the cost of a bag is the maximum size of an
independent set it contains, rather than its cardinality.  Tree-independence
number has since been studied extensively, both structurally and
algorithmically~\cite{DBLP:journals/jctb/DallardMS24a,abrishami2024tree,
dallard2024treewidth,chudnovsky2026tree,HMV25,
DBLP:journals/corr/abs-2601-15861}.
Nevertheless, a drawback of this parameter is that it remains unbounded even
on very simple graphs, namely complete bipartite graphs (bicliques).  Indeed,
$\treealpha(K_{t,t})=t$.

To accommodate bicliques, Yolov~\cite{yolov2018minor} introduced a relaxation
of tree-independence number, which he called \emph{minor-matching
hypertreewidth} and which is now typically referred to as \emph{induced
matching treewidth}, denoted by
$\treemu(G)$~\cite{DBLP:journals/jcss/LimaMMORS26}.
Here the cost of a bag is the largest size of an induced matching such that
every edge of the matching has an endpoint in the bag; we emphasize that the
edges of the matching need not be contained in the bag.
The parameter minimizes the maximum cost over all tree decompositions.
In particular, complete bipartite graphs have induced matching treewidth one.

In our previous work~\cite{nikabadi2026induced}, we generalized these ideas by
introducing \emph{induced packing treewidth}.  Let $\cH$ be a family of graphs.
An induced $\cH$-packing in $G$ is a collection of pairwise anticomplete induced
subgraphs of $G$, each isomorphic to a member of $\cH$.  The induced
$\cH$-packing treewidth of $G$, denoted by $\treepi_{\cH}(G)$, is obtained by minimizing, over all tree
decompositions of $G$, the maximum number of members of such a packing that can
all intersect one bag.
Thus tree-independence number and induced matching
treewidth are precisely the cases $\cH=\{P_1\}$ and $\cH=\{P_2\}$,
respectively.
If $\cH$ is a singleton $\{H\}$, we write $\treepi_H(G)$ instead of $\treepi_{\{H\}}(G)$.

The framework of induced packing treewidth brings together two very actively studied areas of structural and algorithmic graph theory: structured graph decompositions and graph
classes defined by forbidden induced subgraphs.  For a graph $H$, a graph is
\emph{$H$-free} if it contains no induced subgraph isomorphic to $H$; for a
family $\cH$, it is \emph{$\cH$-free} if it is $H$-free for every
$H\in\cH$.
Equivalently, the induced $\cH$-packing treewidth is zero exactly
for $\cH$-free graphs.
Algorithmic applications of boundedness of induced packing treewidth are discussed in the paper introducing this family of parameters~\cite{nikabadi2026induced}.

\subsection{Our results}
In this paper we focus instead on structural questions:
we investigate how induced packing treewidth interacts with classical graph structure and with other width parameters.
\paragraph{Excluding a biclique}
Our first main result concerns the relationship between induced packing
treewidth and tree-independence number.  To put it in context, recall
that $\treemu(G)\leq\treealpha(G)$ for every graph $G$, where
$\treemu$ denotes induced matching treewidth.  Conversely,
tree-independence number cannot be bounded by a function of induced
matching treewidth: $\treemu(K_{t,t})=1$, whereas
$\treealpha(K_{t,t})=t$.  This led to the conjecture that large
bicliques are the only obstruction~\cite{DBLP:journals/jcss/LimaMMORS26};
it was recently confirmed
in~\cite{abrishami2025excluding}:

\begin{theorem}[Abrishami, Briański, Czyżewska, McCarty, Milanič, Rzążewski, and Walczak~\cite{abrishami2025excluding}]\label{thm:treemu-biclique}
For any positive integers $\mu$ and $t$, there is an integer $c(\mu, t)$ such that every $K_{t,t}$-free graph $G$ with $\treemu(G) \leq \mu$ satisfies $\treealpha(G) \leq c(\mu, t)$.
\end{theorem}

Induced packing treewidth gives a broader setting for this question.
For every graph $G$,
$\treepi_{P_3}(G)\leq\treemu(G)\leq\treealpha(G)$; see also the
structural properties in~\cref{lem:structural-package}.  We ask
whether the same biclique obstruction governs the gap between bounded
induced $P_t$-packing treewidth and bounded tree-independence number.
The following theorem answers this question affirmatively and extends
\cref{thm:treemu-biclique}; its proof appears in~\cref{sec:no-Ktt}.

\begin{restatable}{theorem}{goingtotreealpha}\label{thm:going-to-tree-alpha}
For all $a,t\in\mathbb{N}$ and $k\in\mathbb{N}\cup\{0\}$, there is a constant $\zeta=\zeta(a,t,k)$ such that every $K_{a,a}$-free graph $G$ with $\treepi_{P_t}(G)\leq k$ satisfies $\treealpha(G)\leq \zeta$.
\end{restatable}

\Cref{thm:going-to-tree-alpha} has a further consequence.
Recently, Hajebi and Spirkl~\cite{hajebi2026tree} proved that
a hereditary graph class defined by finitely many forbidden
induced subgraphs has bounded tree-independence number if and
only if it is $(\tw,\omega)$-bounded\footnote{A graph class is
$(\tw,\omega)$-bounded if there is a function
$f:\mathbb{N}\to\mathbb{N}$ such that
$\tw(H)\leq f(\omega(H))$ for every non-null induced subgraph
$H$ of every graph in the class.}.
In particular, they confirmed a conjecture of Dallard et al.\
(see~\cite[Conjecture~1.3]{dallard2024treewidth}) that, for all
$a,t\in\mathbb{N}$, $(K_{a,a},P_t)$-free graphs have bounded
tree-independence number.

For $k=0$, \cref{thm:going-to-tree-alpha} recovers this
path-exclusion result, since $\treepi_{P_t}(G)=0$ if and only
if $G$ is $P_t$-free. For arbitrary fixed $k$, our theorem
extends this conclusion by replacing global $P_t$-exclusion
with bounded induced $P_t$-packing treewidth.
It also gives a counterpart of the Hajebi--Spirkl equivalence:
bounded tree-independence number and $(\tw,\omega)$-boundedness
are equivalent for every hereditary class of bounded induced
$P_t$-packing treewidth, without requiring a finite
forbidden-subgraph characterization.

\begin{restatable}{corollary}{treepitwomega}\label{cor:treepi-tw-omega}
Let $t\in\mathbb{N}$, and let $\mathcal{G}$ be a hereditary
class of bounded induced $P_t$-packing treewidth.
The following are equivalent:
\begin{enumerate}
    \item $\mathcal{G}$ has bounded tree-independence number;
    \item $\mathcal{G}$ is $(\tw,\omega)$-bounded;
    \item there exists $a\in\mathbb{N}$ such that every graph
    in $\mathcal{G}$ is $K_{a,a}$-free.
\end{enumerate}
\end{restatable}

The proof of~\cref{cor:treepi-tw-omega} is given in~\cref{sec:no-Ktt}, together with
that of~\cref{thm:going-to-tree-alpha}. The latter combines a recent result of Hajebi and Spirkl~\cite{hajebi2026tree}, the Gy\'arf\'as path argument (\cref{lem:gyarfas-path}), and a result of Chudnovsky, Hajebi,
Lokshtanov, and Spirkl~\cite{chudnovsky2026tree}.

\paragraph{$\chi$-boundedness}

A second conjecture asserts that classes of graphs with bounded
$\treemu$ are $\chi$-bounded; that is, their chromatic number is
bounded by a function of their clique number.  This conjecture was
recently confirmed in~\cite{DBLP:journals/jcss/LimaMMORS26,
abrishami2025excluding}.

\begin{theorem}[Abrishami, Briański, Czyżewska, McCarty, Milanič, Rzążewski, and Walczak~\cite{abrishami2025excluding}]\label{chibounded-treemu}
Every class of graphs with bounded $\treemu$ is $\chi$-bounded.
\end{theorem}
Our second main result extends~\cref{chibounded-treemu} to induced
packing treewidth.  In particular, it shows that bounded induced
$H$-packing treewidth implies $\chi$-boundedness when $H$ is a path
or a star.

\begin{restatable}{theorem}{chibounded}\label{thm:chi-bounded}
Let $t,d\in \mathbb{N}$ and $k\in\mathbb N\cup\{0\}$. The following statements hold:
\begin{itemize}
    \item The class
    $\{G \mid \treepi_{P_t}(G)\leq k\}$
    is $\chi$-bounded.
    \item The class $\{G \mid \treepi_{K_{1,d}}(G)\leq k\}$
    is $\chi$-bounded.
\end{itemize}
\end{restatable} 

\begin{figure}[t]
	\centering
	\includegraphics[width=0.8\linewidth]{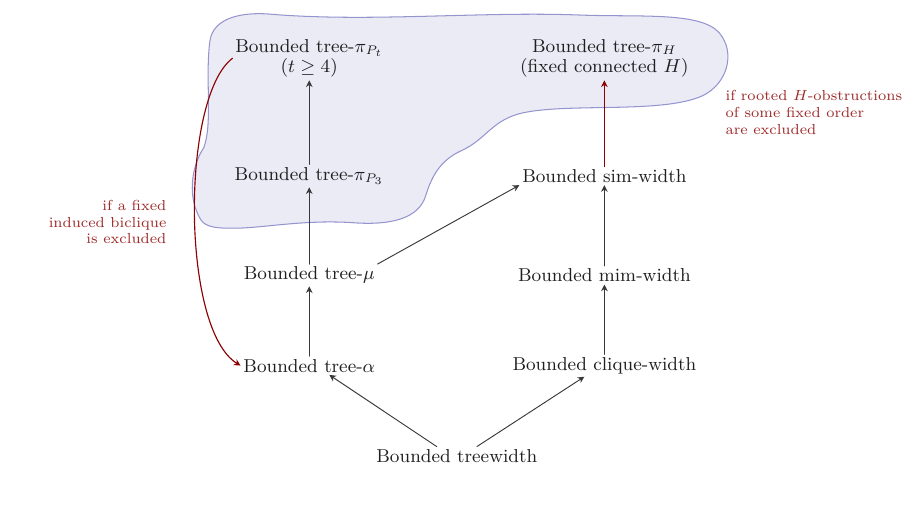}
    \caption{Hierarchy of graph classes related to the ones studied in this paper~\cite{BMPY25}, with newly introduced classes highlighted.}
    \label{fig:hierarchy}
\end{figure}

\Cref{thm:chi-bounded} suggests the following local analogue of the
Gy\'arf\'as--Sumner conjecture, which states that, for every tree $F$,
the class of $F$-free graphs is $\chi$-bounded~\cite{gyarfas1975ramsey,
sumner1981subtrees}.  The
Gy\'arf\'as--Sumner conjecture is known for several families of trees
but remains open in general; see~\cite[Section~3]{scott2020survey} for
a survey.
\begin{conjecture}[Local Gy\'arf\'as--Sumner conjecture]\label{conj:local-gs}
	Let $F$ be a forest and let $k\in\mathbb{N}\cup\{0\}$. There is a function $f_{F,k}$ such that every graph $G$ with $\treepi_F(G)\leq k$ satisfies
	$\chi(G)\leq f_{F,k}(\omega(G))$.
\end{conjecture}

\Cref{thm:chi-bounded} proves~\cref{conj:local-gs} when $F$ is a
path or a star.  The conjecture cannot extend to a graph $F$ that
contains a cycle: $F$-free graphs, equivalently graphs of induced
$F$-packing treewidth zero, are not
$\chi$-bounded~\cite{erdos-girth-chromatic-number}.

\medskip

\paragraph{Erd\H{o}s--Hajnal property}
Our third main result concerns the Erd\H{o}s--Hajnal property of
classes of graphs with bounded induced packing treewidth.  A
hereditary class $\cG$ has the \emph{Erd\H{o}s--Hajnal property} if
there exists $\varepsilon>0$ such that every graph $G\in\cG$ has a
clique or an independent set of size at least $|V(G)|^\varepsilon$.
Erd\H{o}s and Hajnal~\cite{erdos1989ramsey} posed the following
conjecture at the intersection of graph Ramsey theory and structural
graph theory.

\begin{conjecture}[Erd\H{o}s--Hajnal]\label[conjecture]{conj:EH}
For every graph $H$, the class of $H$-free graphs has the
Erd\H{o}s--Hajnal property.
\end{conjecture}

The conjecture predicts a local-to-global phenomenon: excluding one
fixed induced subgraph forces a polynomial-size clique or independent
set.  It remains wide open and appears difficult even for small
graphs $H$.  Recent breakthroughs by Chudnovsky, Scott, Seymour, and
Spirkl~\cite{chudnovsky2023erdHos} and by Nguyen, Scott, and
Seymour~\cite{nguyen2026induced}, together with earlier results,
established \cref{conj:EH} for every graph $H$ on at most five
vertices.

We prove the following.

\begin{restatable}{theorem}{EH}\label{thm:EH}
Let $k\in\mathbb N\cup\{0\}$, and let $H$ be a nonempty graph.
Let $\mathcal{G}_{H,k}:=\{G\mid\treepi_H(G)\leq k\}$.
The class $\mathcal{G}_{H,k}$ has the Erd\H{o}s--Hajnal property if
and only if the class of $H$-free graphs has the Erd\H{o}s--Hajnal
property.
\end{restatable}

\paragraph{Sim-width}
We now turn to our final result, which characterizes the relationship
between \emph{sim-width} and induced $H$-packing treewidth.  We defer
the precise definitions of sim-width to~\cref{sec:sim}; it is a graph
width parameter based on \emph{branch decompositions}, introduced by
Vatshelle~\cite{vatshelle2012new}.

Bounded induced matching treewidth implies bounded sim-width, but the
converse does not hold; see~\cite{DBLP:journals/jcss/LimaMMORS26}.
Further relationships between sim-width and other parameters in
restricted graph classes were studied in~\cite{BMPY25}, see also~\cref{fig:hierarchy}.
The authors
posed the following question:

\begin{question}[Brettell, Munaro, Paulusma, and Yang~\cite{BMPY25}]\label{question:Brettell-etal}
For every $r\in\mathbb N$, does every class of $K_{r,r}$-free graphs with bounded sim-width have bounded tree-independence number?
\end{question}

Subsequent work considered variants of this question.
For example, Štorgel, Choi, Koerts, and Vasić~\cite{vstorgel2026tree} stated the following weakening of \cref{question:Brettell-etal}.

\begin{question}[Štorgel, Choi, Koerts, and Vasić~\cite{vstorgel2026tree}]\label{question:Storgel-etal}
For each $r \geq 2$, does every $K_{1,r}$-free graph class with bounded sim-width have bounded
tree-independence number?
\end{question}

On the other hand, Abrishami, Briański, Czyżewska, McCarty, Milanič,
Rzążewski, and Walczak~\cite{abrishami2025excluding} asked the
following related strengthening of~\cref{question:Brettell-etal}.
For a positive integer $r$, an \emph{$r$-obstruction} is a graph whose
vertex set consists of four independent sets $A$, $B$, $C$, and $D$,
each of size $r$, such that $A\cup B$ and $C\cup D$ each induce a
matching of $r$ edges and $B\cup C$ induces a copy of $K_{r,r}$.

\begin{question}[Abrishami, Briański, Czyżewska, McCarty, Milanič, Rzążewski, and Walczak~\cite{abrishami2025excluding}]\label{question:abrishami-etal}
For every $r\in\mathbb N$, does every class of bounded sim-width
graphs that excludes $r$-obstructions as induced subgraphs have
bounded induced matching treewidth?
\end{question}

An affirmative answer to \cref{question:abrishami-etal} implies an
affirmative answer to \cref{question:Brettell-etal}, and therefore to
\cref{question:Storgel-etal}.  Indeed, every $r$-obstruction contains
an induced copy of $K_{r,r}$ and, by~\cref{thm:treemu-biclique},
bounded induced matching treewidth and bounded tree-independence
number are equivalent for $K_{r,r}$-free graph classes.

To state our final result, we introduce a generalization of
$r$-obstructions, called \emph{rooted $H$-obstructions}.  Their
precise definition is deferred to~\cref{sec:sim}; just as an
$r$-obstruction forces induced matching treewidth at least $r$, a
rooted $H$-obstruction of order $r$ forces induced $H$-packing treewidth at
least $r$.  For $H=K_1$, these obstructions are induced copies of
$K_{r,r}$, and for $H=P_2$, they are precisely the $r$-obstructions
defined above.

We also show that sim-width and induced $P_3$-packing treewidth are
incomparable.  The following theorem gives the promised
characterization and simultaneously resolves
\cref{question:abrishami-etal,question:Brettell-etal,question:Storgel-etal}.

\begin{restatable}{theorem}{hobstruction}
\label{thm:simw-rooted-H-obstruction-restate}
Let $H$ be a connected graph, and let $\mathcal G$ be a class
of bounded sim-width graphs. Then $\mathcal G$ has bounded induced
$H$-packing treewidth if and only if there exists an integer
$r\in \mathbb{N}$ such that no graph in $\mathcal G$ contains a rooted
$H$-obstruction of order $r$ as an induced subgraph.
\end{restatable}

\begin{remark}
Very recently, Niu and Wang~\cite[Theorem~1.2]{niu2026simwidth}
independently showed that, for every fixed integer $r\geq2$, the
tree-independence number of $K_{r,r}$-free graphs is polynomially
bounded in their sim-width.  This answers in particular
\cref{question:Brettell-etal,question:Storgel-etal}.
However, our characterization is not subsumed by their
theorem as stated. Their result assumes the exclusion
of a fixed induced biclique and does not directly address \cref{question:abrishami-etal}.
In particular, excluding $r$-obstructions need not exclude
large induced bicliques.

In contrast, \cref{thm:simw-rooted-H-obstruction-restate}
characterizes bounded induced $H$-packing treewidth within classes of
bounded sim-width for every fixed connected graph $H$, by identifying
the exclusion of rooted $H$-obstructions of some fixed order as a
necessary and sufficient condition.
The case $H=K_1$ recovers the qualitative boundedness
conclusion of Niu and Wang, whereas the case $H=P_2$
answers the stronger $r$-obstruction question.
\end{remark}

\section{Preliminaries}\label{sec:prelim}
For a positive integer $n$, by $[n]$ we denote the set $\{1,\ldots,n\}$.
For a set $X$, by $2^X$ we denote the family of all subsets of $X$. 

\paragraph{Graphs}
Throughout the paper, graphs have finite vertex sets, no loops, and no parallel edges. 
Let $G$ be a~graph with vertex set $V(G)$ and edge set $E(G)$.
For $X \subseteq V(G)$, we denote the subgraph of $G$ \emph{induced by $X$} as $G[X]$,
that is, $G[X] = (X, \{uv \mid u, v \in X \mbox{ and } uv \in E(G) \})$,
and we write $G-X$ for $G[V(G) \setminus X]$.

By $N_G(v)$ we denote the set of neighbors of a vertex $v$ in $G$, and we use $N_G[v]$ for $N_G(v) \cup \{v\}$.
For a set $X \subseteq V(G)$, we define $N_G[X] = \bigcup_{v \in X} N_G[v]$.
We skip the subscript $G$ when it is clear from the context.
Two disjoint sets $X,Y \subseteq V(G)$ are \emph{complete} (resp., \emph{anticomplete}) if all (resp., no) edges between them exist. Analogously, we say that a vertex is (anti)complete to a set.

By $\alpha(G)$ and $\omega(G)$ we denote, respectively, the sizes of a largest independent set and a largest clique in $G$.

When it does not lead to confusion, we will sometimes identify induced subgraphs with their vertex sets. For example, for $X \subseteq V(G)$, we write $\alpha(X)$ for $\alpha(G[X])$.

The length of a path is the number of its edges, i.e., $P_t$ has length $t-1$.
For $X,Y\subseteq V(G)$, an \defn{$(X,Y)$-path} is an induced path $P$
in $G$ such that either $P$ has length zero and $V(P)\subseteq X\cap Y$,
or $P$ has positive length, one end belongs to $X\setminus Y$, the other
belongs to $Y\setminus X$, and the interior of $P$ is disjoint from
$X\cup Y$.

\paragraph{Ramsey's theorem.}
We use the following classical result from Ramsey theory.

\begin{theorem}[Multicolor Ramsey’s theorem~\cite{ramsey1987problem}]
\label{thm:ramsey}
For all $k,s_1,\ldots,s_k \in\mathbb N$, there is an integer $\Ram_k(s_1,\ldots,s_k)$ such that
every $k$-edge-coloring of a complete graph on at least $\Ram_k(s_1,\ldots,s_k)$ vertices contains a copy of $K_{s_i}$ whose edges all have color $i$.
\end{theorem}

For $k=2$, write $\Ram(s,t)=\Ram_2(t,s)$. Thus, every graph on at
least $\Ram(s,t)$ vertices contains either an independent set of size
$s$ or a clique of size $t$.

\paragraph{Layerings.}
Let $G$ be a nonempty graph. A \defn{layering} of $G$ is a sequence
$
    \mathcal L=(L_0,\ldots,L_h),
$
where $h\in\mathbb N\cup\{0\}$ and $L_0,\ldots,L_h$ are pairwise
disjoint nonempty subsets of $V(G)$ whose union is $V(G)$, such that,
for every edge $uv\in E(G)$ with $u\in L_i$ and $v\in L_j$, one has
$
    |i-j|\leq1.
$
The sets $L_0,\ldots,L_h$ are called the \defn{layers} of
$\mathcal L$. In particular, $L_i$ and $L_j$ are anticomplete whenever
$|i-j|\geq2$. For every $i\in\{0,\ldots,h\}$, we write
$
    L_{\leq i}:=\bigcup_{j=0}^i L_j.
$

Suppose that $G$ is connected, and let
$\emptyset\neq R\subseteq V(G)$. For every $v\in V(G)$, set
$
    d_G(v,R):=\min\{d_G(v,r)\mid r\in R\},
$
and let
$
    h:=\max_{v\in V(G)}d_G(v,R).
$
For every $i\in\{0,\ldots,h\}$, set
\begin{equation*}
    L_i:=\{v\in V(G)\mid d_G(v,R)=i\}.
\end{equation*}
Then $(L_0,\ldots,L_h)$ is a layering of $G$, called the
\defn{distance layering} of $G$ from $R$. Indeed, the triangle
inequality gives
$
    |d_G(u,R)-d_G(v,R)|\leq1
$
for every edge $uv\in E(G)$, and a shortest path from a vertex at
distance $h$ from $R$ meets every layer.

Notice that $L_0=R$, and every vertex in $L_i$, for
$i\in\{1,\ldots,h\}$, has a neighbor in $L_{i-1}$. Consequently, if
$G[R]$ is connected, then $G[L_{\leq i}]$ is connected for every
$i\in\{0,\ldots,h\}$. When $R=\{r\}$, we also say that the distance
layering is \defn{rooted at $r$}.

\paragraph{Set systems and hitting sets}
For $b\in\mathbb{N}$, a \defn{$b$-system} in a graph $G$ is a family of nonempty subsets of $V(G)$, each of cardinality at most $b$. A $b$-system is \defn{anticomplete} if every two distinct members are anticomplete in $G$. A set $X\subseteq V(G)$ is a \defn{hitting set} for a family $\calS$ of subsets of $V(G)$ if $X\cap S\neq\emptyset$ for every $S\in\calS$.

The following result of Hajebi and Spirkl~\cite{hajebi2026tree}, used by them to
characterize the hereditary classes defined by finitely many excluded induced
subgraphs that have bounded $\treealpha$, is applied twice in this paper. In both
applications, its alternative outcome is ruled out by a boundedness assumption
(on induced packings in~\cref{sec:no-Ktt}, on sim-matchings in~\cref{sec:sim}),
so that it yields a hitting set of bounded independence number.

\begin{lemma}[Hajebi and Spirkl~\cite{hajebi2026tree}]\label{lem:hit-vs-anti}
For all $a,b,q\in\mathbb{N}$, there is a constant
$c_{\ref{lem:hit-vs-anti}}=c_{\ref{lem:hit-vs-anti}}(a,b,q)\in\mathbb{N}$
such that, for every $K_{a,a}$-free graph $G$ and every $b$-system $\calS$ in $G$, one of the following holds:
\begin{enumerate}[label=\textup{(\alph*)}, leftmargin=8mm]
    \item there is a hitting set $X\subseteq V(G)$ for $\calS$ with
    $\alpha(G[X])<c_{\ref{lem:hit-vs-anti}}$; or
    \item there is an anticomplete $b$-system $\calQ\subseteq\calS$ with
    $|\calQ|\geq q$.
\end{enumerate}
\end{lemma}

\paragraph{Weighted graphs and balanced separators.}
A \emph{vertex-weighted graph} is a pair $(G,\wei)$ where $G$ is a graph and $\wei:V(G)\to\mathbb Q_{\ge0}$ is a weight function; we also call such a $\wei$ a \defn{vertex-weighting} of $G$. For a set $X\subseteq V(G)$, we write $\wei(X)=\sum_{v\in X}\wei(v)$, and we write $\wei(G):=\wei(V(G))$.
We assume that all computations on weights can be performed in constant time.

For a weighted graph $(G,\wei)$, a \emph{balanced separator} of $G$ is a set $S \subseteq V(G)$ such that every component of $G-S$ has weight at most $\wei(G)/2$. We will often work with balanced separators for \emph{uniform} weight functions, i.e., $\wei(v)=1$ for every $v\in V(G)$. In this case, every component of $G-S$ has at most $|V(G)|/2$ vertices.

\paragraph{Tree decompositions and balanced separators}
For a graph $G$, a \emph{tree decomposition} of $G$ is a pair $(T,\beta)$, where $T$ is a tree and $\beta \colon V(T) \rightarrow 2^{V(G)}$ is a map with the following properties:
\begin{itemize}
    \item For every $v \in V(G)$, there exists $x \in V(T)$ such that $v \in \beta(x)$.
    \item For every $uv \in E(G)$, there exists $x \in V(T)$ such that $u, v \in \beta(x)$.
    \item For every $v \in V(G)$, the subgraph of $T$ induced by $\{x \in V(T) \mid v \in \beta(x)\}$ is connected.
\end{itemize}
For every $x \in V(T)$, we refer to $x$ as a \emph{node} of $T$, and to $\beta(x)$ as a \emph{bag} of $(T, \beta)$.
The \emph{width} of a tree decomposition $(T, \beta)$ is $\max_{x\in V(T)} |\beta(x)|-1$.
The \emph{treewidth} of $G$, denoted by $\tw(G)$, is the minimum width of a tree decomposition of $G$.
The \emph{independence number} of a tree decomposition $(T, \beta)$ of $G$ is $\max_{x\in V(T)} \alpha(G[\beta(x)])$. The \emph{tree-independence number} of $G$, denoted $\treealpha(G)$, is the minimum independence number of a tree decomposition of $G$.

The following lemma is a standard fact about tree decompositions and balanced separators (e.g.~\cite[Lemma~7.19]{cygan2015parameterized}).

\begin{lemma}\label{lem:balanced-bag}
Let $(G,\wei)$ be a weighted graph with a tree decomposition $(T,\beta)$.
Then there is a node $x\in V(T)$ such that
$\beta(x)$ is a balanced separator of $(G,\wei)$.
\end{lemma}

Conversely, balanced separators of bounded independence number, found for
\emph{every} vertex-weighting, certify bounded tree-independence number. The
following is the specialization of~\cite[Lemma~7.1]{chudnovsky2026tree} to balanced
separators. That result is stated for normalized weight functions, i.e., those with
$\wei(G)=1$; since the conclusion is vacuous when $\wei(G)=0$, scaling the weights
gives the equivalent formulation below.

\begin{lemma}[Chudnovsky, Hajebi, Lokshtanov, and Spirkl~\cite{chudnovsky2026tree}]\label{lem:sep-vs-ta}
Let $G$ be a graph and let $\lambda\in\mathbb{N}$. Suppose that, for every vertex-weighting $\wei$ of $G$, the weighted graph $(G,\wei)$ has a balanced separator $Z$ satisfying $\alpha(G[Z])\leq \lambda$. Then
$
    \treealpha(G)\leq 5\lambda.
$
\end{lemma}

\paragraph{$B$-rooted subgraphs and $(B,\cH)$-clean sets}
Let $G$ be a graph and $B \subseteq V(G)$.
Any induced subgraph of $G$ whose vertex set intersects $B$ is called \emph{$B$-rooted}.

Let $\cH$ be a family of graphs.
An \emph{induced $\cH$-packing} in $G$ is a set of pairwise anticomplete induced subgraphs of $G$, each isomorphic to a member of $\cH$.
An induced $\cH$-packing is \emph{$B$-rooted} if every member of the packing intersects $B$ (see~\cref{fig:B-rooted}).

\begin{figure}[t]
    \centering
    \includegraphics[width=0.45\linewidth]{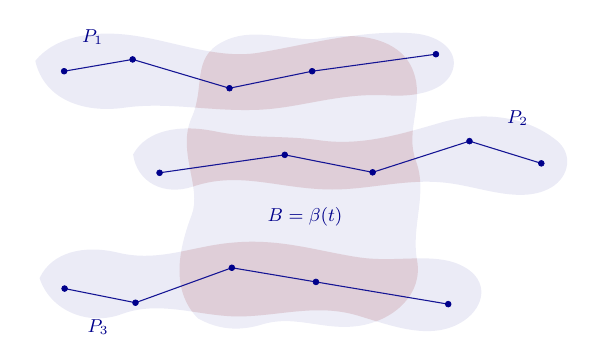}
    \caption{$B$-rooted induced $\{P_5\}$-packing.}
    \label{fig:B-rooted}
\end{figure}

We say a set $C\subseteq V(G)$ is \emph{$(B,\cH)$-clean} if
$G[C]$ has no $B$-rooted induced subgraph isomorphic to any member of $\cH$.
If $V(G)$ is $(B,\cH)$-clean, we say that $G$ is $(B,\cH)$-clean.

The size of the largest $B$-rooted induced $\cH$-packing in $G$ is denoted by $\pi_{\cH}(G,B)$.
For a tree decomposition $\cT=(T,\beta)$ of $G$, we define $\pi_{\cH}(G,\cT) = \max_{x \in V(T)} \pi_{\cH}(G, \beta(x))$.
Then, $\treepi_{\cH}(G) = \min_{\cT} \pi_{\cH}(G,\cT)$ where the minimum is taken over all tree decompositions $\cT$ of $G$.
When $\cH$ is a singleton $\{H\}$, we abbreviate the notation above by writing $H$ instead of $\{H\}$.

We also collect the following basic properties of induced packing treewidth which are proved in our earlier work:

\begin{lemma}[Nikabadi and Rzążewski~\cite{nikabadi2026induced}]\label{lem:structural-package}
Let $\cH,\cH'$ be families of graphs and $G,G'$ be graphs. The following hold.
\begin{enumerate}
\item\label{sprop:3} If $G'$ is an induced subgraph of $G$, then $\treepi_{\cH}(G')\leq \treepi_{\cH}(G)$.

\item\label{sprop:4} If $\cH'\subseteq \cH$, then $\treepi_{\cH'}(G)\leq \treepi_{\cH}(G)$.

\item\label{sprop:7} For every $t\geq 1$, one has $\treepi_{P_{t+1}}(G)\leq \treepi_{P_t}(G)$.

\end{enumerate}
\end{lemma}

\section{Induced path-packing treewidth in graphs with no large bicliques}\label{sec:no-Ktt}
In this section, we prove~\cref{thm:going-to-tree-alpha}.

\goingtotreealpha*

Let us first explain the proof strategy.
By~\cref{lem:sep-vs-ta}, it suffices to show that, for every vertex-weighting $\wei$ of $G$, the weighted graph $(G,\wei)$ has a balanced separator whose independence number is bounded in terms of $a,t,k$ only.
So let us fix a vertex-weighting $\wei$.
The proof then consists of the following steps.
\begin{enumerate}[{Step }1.]
    \item We use~\cref{lem:balanced-bag} to find a bag $B$ of a tree decomposition witnessing $\treepi_{P_t}(G)\leq k$ that is a balanced separator of $(G,\wei)$.
    \item Being such a bag, $B$ satisfies $\pi_{P_t}(G,B)\leq k$, and so~\cref{lem:hit-vs-anti} provides a set $X$ of bounded independence number meeting every induced copy of $P_t$ that intersects $B$.
    Writing $G'=G\setminus X$ and $B'=B\setminus X$, the graph $G'$ is thus $(B',P_t)$-clean, and, after adjusting the weight function to some $\wei'$, the set $B'$ is a balanced separator of $(G',\wei')$.

    \item This is the main technical step.
    We show that in a $(B',P_t)$-clean graph that excludes \emph{pure subdivisions of a large biclique}, the balanced separator $B'$ can be traded for a balanced separator $Z$ of bounded independence number (\cref{lem:clean-root-separator}).
    The hypothesis on subdivisions is supplied by our assumptions: a $K_{a,a}$-free graph of bounded induced $P_t$-packing treewidth contains no pure subdivision of a large biclique (\cref{lem:exclude-pure-subdivision}).

    \item Finally, $X\cup Z$ is a balanced separator of $(G,\wei)$ whose independence number is bounded by a constant depending only on $a,t,k$, as required.
\end{enumerate}

\bigskip
Before we proceed to the proof, we need to introduce some intermediate results.
\subsection{Dominated separators and subdivisions of bicliques}

For $s\in\mathbb{N}$, we say that a graph $G$ has \defn{$s$-dominated balanced separators} if, for every vertex-weighting $\wei$ of $G$, there is a set $X\subseteq V(G)$ with $|X|\leq s$ such that $N[X]$ is a balanced separator of $(G,\wei)$.
The classical source of such separators is the Gyárfás path argument~\cite{gyarfas1987problems}; see also~\cite[Lemma~5.3]{chudnovsky2024quasi} or~\cite[Lemma~3.2]{chudnovsky2025dominated}.

\begin{lemma}\label{lem:gyarfas-path}
Let $(G,\wei)$ be a connected vertex-weighted graph, and let $v\in V(G)$.
There is an induced path $P$ starting at $v$ such that every component $C$ of $G\setminus N[P]$ satisfies $\wei(C)\leq \wei(G)/2$.
\end{lemma}

\Cref{lem:gyarfas-path} was in particular used to show the following.

\begin{lemma}[Nikabadi and Rzążewski~\cite{nikabadi2026induced}]\label{lem:treepi-Pt-dominated-separator}
For every $t\in\mathbb{N}$ and $k\in\mathbb{N}\cup\{0\}$, every graph $G$ satisfying $\treepi_{P_t}(G)\leq k$ has $(k+1)t$-dominated balanced separators.
\end{lemma}

Let us first recall the relevant terminology.
Let $H$ and $H'$ be graphs. We say that $H'$ is a \defn{subdivision} of $H$ if, up to isomorphism, $H'$ can be obtained from $H$ by replacing every edge $uv\in E(H)$ with a $u$-$v$ path, in such a way that the introduced paths are pairwise internally vertex-disjoint and no internal vertex of such a path belongs to $V(H)$.
The vertices of $H'$ corresponding to the vertices of $H$ are called the \defn{branch vertices}, and the paths replacing the edges of $H$ are called the \defn{replacement paths}.
Following~\cite{hajebi2026tree}, we say that a subdivision $H'$ of $H$ is
\begin{itemize}
    \item a \defn{$(\leq \ell)$-subdivision} of $H$, where $\ell\in\mathbb{N}$, if every replacement path has at most $\ell$ vertices (i.e., length at most $\ell-1$);
    \item a \defn{proper subdivision} of $H$ if every replacement path has length at least two; and
    \item a \defn{pure subdivision} of $H$ if either $H'$ is isomorphic to $H$ or $H'$ is a proper subdivision of $H$.
\end{itemize}
A \defn{pure $(\leq \ell)$-subdivision} of $H$ is a subdivision of $H$ that is both pure and a $(\leq \ell)$-subdivision of $H$.

In the next lemma, we show that deleting few closed neighborhoods cannot disconnect a large subdivided biclique.

\begin{lemma}\label{lem:exclude-proper-subdivision}
Let $s\in\mathbb{N}$ and let $p\geq 2s+1$.
No proper subdivision of $K_{p,p}$ has $s$-dominated balanced separators.
\end{lemma}

\begin{proof}
Let $F$ be a proper subdivision of $K_{p,p}$.
Let $A,B$ be the partition of branch vertices of $F$ corresponding to the two sides of $K_{p,p}$, and denote by $Q_{u,v}$ the replacement path joining $u\in A$ to $v\in B$.
Give every vertex of $A$ weight $1$ and every other vertex of $F$ weight $0$; call this weighting $\wei$, so that $\wei(F)=p$.
Suppose for a contradiction that there is a set $X\subseteq V(F)$ with $|X|\leq s$ such that $N[X]$ is a balanced separator of $(F,\wei)$; that is, every component of $F-N[X]$ contains at most $p/2$ vertices of $A$.

Since the subdivision is proper, every replacement path has length at least two.
Hence no two branch vertices of $F$ are adjacent, and every internal vertex of a replacement path $Q$ has both of its neighbors on $Q$.
This gives, for every $x\in V(F)$:
\begin{enumerate}
\item[(i)] $N[x]$ contains at most one vertex of $A$ and at most one vertex of $B$; and
\item[(ii)] there is at most one pair $(u,v)\in A\times B$ such that $u,v\notin N[x]$ and $N[x]$ intersects $V(Q_{u,v})$.
\end{enumerate}

Put $A'=A\setminus N[X]$ and $B'=B\setminus N[X]$.
By (i), each $x\in X$ removes at most one vertex from each of $A$ and $B$, so
\[
    |A'|,\,|B'|\ \geq\ p-|X|\ \geq\ p-s.
\]
Since $p\geq 2s+1$, we have $p-s\geq s+1$ and $p-s>p/2$.

Call a pair $(u,v)\in A'\times B'$ \emph{broken} if $V(Q_{u,v})$ meets $N[X]$.
If $(u,v)$ is broken, then there is $x \in X$ such that $N[x]$ meets $V(Q_{u,v})$; moreover $u,v \notin N[x]$, since $u,v\notin N[X]$.
Thus, by (ii), for each $x \in X$ there is at most one pair $(u,v)$ that became broken due to $x$.
Hence there are at most $|X|\leq s$ broken pairs in total.

Now let $u,u'\in A'$ be arbitrary.
As at most $s$ pairs are broken and $|B'|\geq s+1$, there is some $v\in B'$ for which neither $(u,v)$ nor $(u',v)$ is broken.
The paths $Q_{u,v}$ and $Q_{u',v}$ then avoid $N[X]$ and share the vertex $v$, so $u$ and $u'$ lie in the same component of $F-N[X]$.
As $u$ and $u'$ were arbitrary, all of $A'$ lies in a single component of $F-N[X]$, whose weight is thus at least $|A'|\geq p-s>p/2$.
This contradicts the assumption that $N[X]$ is a balanced separator of $(F,\wei)$.
\end{proof}

Now, combining~\cref{lem:treepi-Pt-dominated-separator,lem:exclude-proper-subdivision}, we immediately obtain the following result.

\begin{lemma}\label{lem:exclude-pure-subdivision}
Let $a,t\in\mathbb{N}$ and $k\in\mathbb{N}\cup\{0\}$, and let $p\geq\max\{a,\,2(k+1)t+1\}$.
Then no $K_{a,a}$-free graph $G$ with $\treepi_{P_t}(G)\leq k$ contains an induced subgraph isomorphic to a pure subdivision of $K_{p,p}$.
\end{lemma}

\begin{proof}
Suppose that some induced subgraph $F$ of $G$ is a pure subdivision of $K_{p,p}$.
Since $p\geq a$, the graph $K_{p,p}$ contains an induced copy of $K_{a,a}$, and so $F$ is not isomorphic to $K_{p,p}$ as $G$ is $K_{a,a}$-free.
Thus $F$ is a proper subdivision of $K_{p,p}$.
On the other hand, $\treepi_{P_t}(F)\leq\treepi_{P_t}(G)\leq k$ by~\cref{lem:structural-package}\eqref{sprop:3}, so~\cref{lem:treepi-Pt-dominated-separator} implies that $F$ has $s$-dominated balanced separators for $s=(k+1)t$.
As $p\geq 2s+1$, this contradicts~\cref{lem:exclude-proper-subdivision}.
\end{proof}

\subsection{Balanced separators in clean graphs}

This subsection is devoted to the main technical step, stated as~\cref{lem:clean-root-separator} below: in a $(B,P_t)$-clean graph excluding a pure subdivision of a large biclique, a balanced separator $B$ can be replaced by a balanced separator of bounded independence number.

We will need the result of Hajebi and Spirkl~\cite{hajebi2026tree}; it is a refinement of~\cref{lem:hit-vs-anti}: if the system itself has no hitting set of bounded independence number, then it produces two members for which the family of all \emph{short} paths between them has one.
For $\ell\in\mathbb{N}$, we denote by $\calP_G^{\leq \ell}(X,Y)$ the family of all $(X,Y)$-paths in $G$ with at most $\ell$ vertices.

\begin{lemma}[Hajebi and Spirkl~\cite{hajebi2026tree}]\label{lem:hit-vs-far}
For all $a,b,q,\ell\in\mathbb{N}$, there are constants
$
    c_{\ref{lem:hit-vs-far}}
    =c_{\ref{lem:hit-vs-far}}(a,b,q,\ell)\in\mathbb{N}$ and
    $d_{\ref{lem:hit-vs-far}}=d_{\ref{lem:hit-vs-far}}(a,b,q,\ell)\in\mathbb{N}$
such that, for every graph $G$ with no induced subgraph isomorphic to a pure $(\leq \ell)$-subdivision of $K_{a,a}$ and every $b$-system $\calS$ in $G$, one of the following holds:
\begin{enumerate}[label=\textup{(\alph*)}, leftmargin=8mm]
    \item there is a hitting set $X\subseteq V(G)$ for $\calS$ with
    $\alpha(G[X])<c_{\ref{lem:hit-vs-far}}$; or
    \item there is a $b$-system $\calQ\subseteq\calS$ with $|\calQ|\geq q$ such that, for all distinct $S_1,S_2\in\calQ$, the family $\calP_G^{\leq \ell}(S_1,S_2)$ has a hitting set
    $
        Y\subseteq V(G)\setminus(S_1\cup S_2)
    $
    with $\alpha(G[Y])<d_{\ref{lem:hit-vs-far}}$.
\end{enumerate}
\end{lemma}

We will only ever apply~\cref{lem:hit-vs-far} with $q=2$ and with the two bounds $b$ and $\ell$ set to a common value, so, for all $p,t\in\mathbb{N}$, we abbreviate
\[
    c(p,t)=c_{\ref{lem:hit-vs-far}}(p,t,2,t)
    \qquad\text{and}\qquad
    d(p,t)=d_{\ref{lem:hit-vs-far}}(p,t,2,t).
\]

Now, we proceed to the main lemma of this subsection.

\begin{lemma}\label{lem:clean-root-separator}
Let $p\in\mathbb{N}$, let $t\geq2$, and let $(G,\wei)$ be a vertex-weighted graph with no induced subgraph isomorphic to a pure $(\leq t)$-subdivision of $K_{p,p}$.
Let $B\subseteq V(G)$ be a balanced separator of $(G,\wei)$ such that $G$ is $(B,P_t)$-clean.
Then $(G,\wei)$ has a balanced separator $Z$ with
$
    \alpha(Z)<2c(p,t)+d(p,t).
$
\end{lemma}

\begin{proof}
Call an induced path $S$ in $G$ \defn{dominant} if
\begin{itemize}
    \item $S$ has at most $t-1$ vertices and $S\cap B\neq\emptyset$; and
    \item there is a set $X_S\subseteq V(G)$ with $\alpha(X_S)<c(p,t)$ such that $N_G[S]\cup X_S$ is a balanced separator of $(G,\wei)$.
\end{itemize}
Note that the assumption $t\geq2$ is what makes the first condition satisfiable.
Let $\calS$ be the family of vertex sets of all dominant paths; since every member of $\calS$ is nonempty and has at most $t-1$ vertices, $\calS$ is a $t$-system in $G$.
So we may apply~\cref{lem:hit-vs-far} to $\calS$ with parameters $a=p$, $b=t$, $q=2$, and $\ell=t$.
We distinguish two cases (a) and (b), depending on which outcome of~\cref{lem:hit-vs-far} holds.

\medskip
\noindent{\textbf{Case (a).}} Suppose there is a hitting set $X$ for $\calS$ with $\alpha(X)<c(p,t)$.

\begin{claim}
 $X$ is a balanced separator of $(G,\wei)$.
\end{claim}
\begin{claimproof}
Suppose not and let $C$ be the unique component of $G - X$ with $\wei(C) > \wei(G)/2$.
Then $V(C)\cap B\neq\emptyset$: otherwise $C$, being connected and disjoint from $B$, would be contained in a component of $G-B$, contradicting that $B$ is a balanced separator of $(G,\wei)$.
So let $v \in V(C)\cap B$ and apply \cref{lem:gyarfas-path} to $(C,\wei)$ and $v$ to obtain an induced path $P$ of $C$ starting at $v$ such that every component of $C-N_C[P]$ has weight at most $\wei(C)/2\leq\wei(G)/2$.
We claim that $N_G[P] \cup X$ is a balanced separator of $(G,\wei)$.
Indeed, every component of $G-(N_G[P]\cup X)$ is contained either in a component of $G-X$ distinct from $C$, of weight at most $\wei(G)-\wei(C)<\wei(G)/2$, or in $C\setminus N_G[P]\subseteq C-N_C[P]$, and hence in a component of $C-N_C[P]$.

Since $P$ is induced, has an end in $B$, and $G$ is $(B,P_t)$-clean, we have $1\leq|V(P)|\leq t-1$: otherwise the first $t$ vertices of $P$ would induce a copy of $P_t$ meeting $B$.
Hence $P$ is dominant, witnessed by $X_{P}=X$.
But $P$ is a path in $G-X$, so it is disjoint from $X$.
This contradicts the fact that $X$ hits $\calS$.
\end{claimproof}

As $\alpha(X)<c(p,t)\leq2c(p,t)+d(p,t)$, we can set $Z=X$ and the proof of Case (a) is complete.

\medskip
\noindent{\textbf{Case (b).}}
Suppose now that there are two distinct members $S_1,S_2$ of $\calS$ and a set $Y\subseteq V(G)\setminus(S_1\cup S_2)$ with $\alpha(Y)<d(p,t)$ that hits $\calP_G^{\leq t}(S_1,S_2)$.
Note that $S_1\cap S_2=\emptyset$: a vertex of $S_1\cap S_2$ would form an $(S_1,S_2)$-path of length $0$, which $Y$ cannot hit.
For $i\in\{1,2\}$, let $X_i=X_{S_i}$ be as in the definition of a dominant path, and put
$
    Z=X_1\cup X_2\cup Y,
$
so that
\[
    \alpha(Z)
    \leq \alpha(X_1)+\alpha(X_2)+\alpha(Y)
    <2c(p,t)+d(p,t).
\]

We are left with proving the following claim.

\begin{claim}
 $Z$ is a balanced separator of $(G,\wei)$.
\end{claim}
\begin{claimproof}
For contradiction, suppose otherwise and let $C$ be the unique component of $G-Z$ with $\wei(C)>\wei(G)/2$.
Suppose first that $C \cap N_G[S_i]=\emptyset$ for some $i\in\{1,2\}$.
As $C$ is also disjoint from $X_i\subseteq Z$, the graph $C$ is then contained in a component of $G\setminus(N_G[S_i]\cup X_i)$, and hence $\wei(C)\leq \wei(G)/2$,
a contradiction.

So assume that $C$ meets both $N_G[S_1]$ and $N_G[S_2]$ and define $G' = G[C \cup S_1\cup S_2]$.
Note that $G'$ is connected, since each of the sets $C,S_1,S_2$ induces a connected subgraph of $G$ and $C$ meets both $N_G[S_1]$ and $N_G[S_2]$.

We claim that there is a path $Q'\in \calP_G^{\leq t}(S_1,S_2)$ with $V(Q')\subseteq V(G')$.
Recall that $S_1$ intersects $B$.
Let $Q$ be a shortest path in $G'$ from $S_1\cap B$ to $S_2$; it exists because $G'$ is connected.
Since $Q$ is shortest, it is induced in $G'$ and hence in $G$, and its last vertex is its only vertex in $S_2$.
Since $Q$ has an end in $B$ and $G$ is $(B,P_t)$-clean, $Q$ has at most $t-1$ vertices.

We obtain $Q'$ by (possibly) removing some prefix of $Q$, so that $Q'$ starts at the last vertex of $S_1$ appearing on $Q$.
The interior of $Q'$ is then disjoint from $S_1$ by this choice, and from $S_2$ because only the last vertex of $Q$ lies in $S_2$.
Consequently, $Q'$ is an $(S_1,S_2)$-path with at most $t-1$ vertices, and it is induced in $G$ because $Q$ is.

Now, by the definition of $Y$, it intersects $Q'$, which is impossible: $Y$ is disjoint from $S_1\cup S_2$ by its choice, and from $C$ because $Y\subseteq Z$.
Hence this case does not occur, and $Z$ is a balanced separator of $(G,\wei)$.
\end{claimproof}

This completes the proof of the lemma.
\end{proof}

\subsection{Proofs of the main results}
Finally, we are ready to prove~\cref{thm:going-to-tree-alpha}.
\begin{proof}[Proof of~\cref{thm:going-to-tree-alpha}]
We may assume that $t\geq2$: a $B$-rooted induced $P_1$-packing in a graph $G$ is precisely an independent set in $G[B]$, so $\treepi_{P_1}(G)=\treealpha(G)$ and the case $t=1$ holds with $\zeta(a,1,k)=k$.

Set
\[
    p=\max\{a,\,2(k+1)t+1\},
    \qquad
    \eta=c_{\ref{lem:hit-vs-anti}}(a,t,k+1),
    \qquad
    c=c(p,t),
    \qquad
    d=d(p,t),
\]
and let us show that $\zeta=5(\eta+2c+d)$ is sufficient.
By~\cref{lem:sep-vs-ta}, it is enough to prove that, for every vertex-weighting $\wei$ of $G$, the weighted graph $(G,\wei)$ has a balanced separator $Z^*$ with $\alpha(G[Z^*])\leq\eta+2c+d$.

So let $\wei$ be a vertex-weighting of $G$, and fix a tree decomposition $\calT=(T,\beta)$ of $G$ with
$
    \pi_{P_t}(G,\calT)\leq k.
$
By~\cref{lem:balanced-bag}, there is a node $x\in V(T)$ such that $B=\beta(x)$ is a balanced separator of $(G,\wei)$; in particular,
$
    \pi_{P_t}(G,B)\leq k.
$
Let $\calP$ be the $t$-system consisting of the vertex sets of all induced copies of $P_t$ in $G$ that intersect $B$, and apply~\cref{lem:hit-vs-anti} to $\calP$ with parameters $a,t,k+1$.
An anticomplete subsystem of $\calP$ of cardinality at least $k+1$ would be a $B$-rooted induced $P_t$-packing of size at least $k+1$, contrary to $\pi_{P_t}(G,B)\leq k$.
Hence there is a hitting set $X\subseteq V(G)$ for $\calP$ with
$
    \alpha(G[X])<\eta.
$

Let $G'=G\setminus X$ and $B'=B\setminus X$.
Then $G'$ is $(B',P_t)$-clean, since every induced copy of $P_t$ in $G'$ meeting $B'$ would be a member of $\calP$ avoiding $X$.
Moreover,
$
    G'\setminus B'=G\setminus(B\cup X),
$
so every component of $G'\setminus B'$ is contained in a component of $G\setminus B$ and therefore has weight at most $\wei(G)/2$.
If $B'=\emptyset$, then $X$ itself is a balanced separator of $(G,\wei)$ with $\alpha(G[X])<\eta$, and we are done.
Thus we may assume that $B'\neq\emptyset$.

Choose $v_0\in B'$ and define a vertex-weighting $\wei'$ of $G'$ by
$
    \wei'(v_0)=\wei(v_0)+\wei(X)
$
and $\wei'(v)=\wei(v)$ for every $v\in V(G')\setminus\{v_0\}$, so that
$
    \wei'(G')=\wei(G).
$
Every component of $G'\setminus B'$ avoids $v_0$, and hence its $\wei'$-weight equals its $\wei$-weight; consequently $B'$ is a balanced separator of $(G',\wei')$.

By~\cref{lem:exclude-pure-subdivision}, the graph $G$, and hence also its induced subgraph $G'$, contains no induced subgraph isomorphic to a pure subdivision of $K_{p,p}$; in particular, none isomorphic to a pure $(\leq t)$-subdivision of $K_{p,p}$.
Applying~\cref{lem:clean-root-separator} to $(G',\wei')$ and $B'$, we obtain a balanced separator $Z$ of $(G',\wei')$ with
$
    \alpha(G'[Z])<2c+d.
$
Since $\wei(v)\leq\wei'(v)$ for every $v\in V(G')$ and $\wei'(G')=\wei(G)$, every component $C$ of $G'\setminus Z$ satisfies
\[
    \wei(C)
    \leq \wei'(C)
    \leq \frac{\wei'(G')}{2}
    =\frac{\wei(G)}{2}.
\]
As
$
    G\setminus(X\cup Z)=G'\setminus Z,
$
the set $Z^*=X\cup Z$ is therefore a balanced separator of $(G,\wei)$, and
\[
    \alpha(G[Z^*])
    \leq \alpha(G[X])+\alpha(G'[Z])
    \leq \eta+2c+d,
\]
as required.
\end{proof}

\paragraph{Consequences.}
We conclude the section with two consequences of~\cref{thm:going-to-tree-alpha}.
The first one is~\cref{cor:treepi-tw-omega}, announced in the introduction, which we restate.

\treepitwomega*

\begin{proof}
The implication $(1)\Rightarrow(2)$ is a simple consequence of Ramsey's
theorem, as observed by Dallard, Milani\v{c}, and \v{S}torgel~\cite{dallard2024treewidthII}.

For $(2)\Rightarrow(3)$, let $f$ be a
$(\tw,\omega)$-bounding function for $\mathcal{G}$, and choose
an integer $a>\max\{1,f(2)\}$.
If some graph in $\mathcal{G}$ contained an induced copy of
$K_{a,a}$, then heredity would imply that
$K_{a,a}\in\mathcal{G}$. This is impossible, since
$\tw(K_{a,a})=a>f(2)$ and $\omega(K_{a,a})=2$.
Thus, every graph in $\mathcal{G}$ is $K_{a,a}$-free.

For $(3)\Rightarrow(1)$, choose
$k\in\mathbb{N}\cup\{0\}$ such that
$\treepi_{P_t}(G)\leq k$ for every $G\in\mathcal{G}$.
By~\cref{thm:going-to-tree-alpha}, there is a constant
$\zeta=\zeta(a,t,k)$ such that
$\treealpha(G)\leq\zeta$ for every $G\in\mathcal{G}$.
\end{proof}

The second consequence is immediate from~\cref{lem:structural-package}\eqref{sprop:4}: if a family
$\cH$ contains some path $P_t$, then $\treepi_{P_t}(G)\leq\treepi_{\cH}(G)$ for every
graph $G$, so~\cref{thm:going-to-tree-alpha} applies verbatim.

\begin{corollary}\label{cor:family-containing-a-path}
Let $\cH$ be a family of graphs containing at least one path, and let
$a\in\mathbb{N}$ and $k\in\mathbb{N}\cup\{0\}$. There is a constant
$\zeta'=\zeta'(\cH,a,k)$ such that every $K_{a,a}$-free graph $G$ with
$\treepi_{\cH}(G)\leq k$ satisfies $\treealpha(G)\leq\zeta'$.
\end{corollary}

\section[Chi-boundedness and Erd\H{o}s--Hajnal property]{$\chi$-boundedness and Erd\H{o}s--Hajnal property}\label{sec:chi-bounded}
In this section, we prove~\cref{thm:chi-bounded,thm:EH}.
We use without further mention that $\treepi_H$ is monotone under taking induced subgraphs (see~(\ref{sprop:3}) in~\cref{lem:structural-package}), and that $\pi_H(G,B')\leq\pi_H(G,B)$ whenever $B'\subseteq B\subseteq V(G)$, which is immediate from the definition.

\subsection{Induced path-packing treewidth}
We first prove the first assertion of~\cref{thm:chi-bounded}. The case
$t=1$ follows from the standard Ramsey-theoretic bound for
tree-independence number; see, for example,~\cite[Lemma~3.2]{dallard2024treewidth}.
Indeed, let $G$ be a graph with
$\treepi_{P_1}(G)=\ta(G)\leq k$ and $\omega(G)\leq r$, and let
$\mathcal T=(T,\beta)$ be a tree decomposition of $G$ such that
$\alpha(\beta(x))\leq k$ for every $x\in V(T)$. Since
$\omega(\beta(x))\leq r$, Ramsey's theorem~(\cref{thm:ramsey}) gives
$|\beta(x)|\leq \Ram(k+1,r+1)-1$ for every $x\in V(T)$. Consequently,
$\tw(G)\leq \Ram(k+1,r+1)-2$, and therefore
$\chi(G)\leq \Ram(k+1,r+1)-1$.
Thus, throughout the remainder of this subsection, we fix an integer
$t\geq2$.

Gyárfás~\cite[Theorem~2.4]{gyarfas1987problems} proved that for $t \geq 2$, every $P_t$-free graph $G$ satisfies
$\chi(G)\leq g^{\mathsf{path}}_t(\omega(G))$, where
\[
	g^{\mathsf{path}}_t(r)=(t-1)^{r-1}.
\]

We shall repeatedly use the following localized form of the
Gyárfás path construction underlying the proof of~\cite[Theorem~2.4]{gyarfas1987problems}.
Its proof follows the same
induction as Chudnovsky, Scott, Seymour, and
Spirkl~\cite[Lemma~2.3]{chudnovsky2020longoddholes}.

\begin{lemma}\label{lem:gyarfas-extension-local}
Let $a\in\mathbb{N}$ and $q\in\mathbb{N}\cup\{0\}$. Let $G$ be a graph,
let $C\subseteq V(G)$ be such that $G[C]$ is connected, and let
$x_0\in V(G)\setminus C$ have a neighbor in $C$. Suppose that
$\chi(G[N_G(z)\cap C])\leq a$ for every $z\in C\cup\{x_0\}$ and that
$\chi(G[C])>qa$. Then there is an induced path $x_0-x_1-\cdots-x_q$ in
$G$ with $x_1,\ldots,x_q\in C$, and a set
$C'\subseteq C\setminus\{x_1,\ldots,x_q\}$, such that
\begin{itemize}
    \item $G[C']$ is connected;
    \item $x_q$ has a neighbor in $C'$;
    \item $x_0,\ldots,x_{q-1}$ are anticomplete to $C'$; and
    \item $\chi(G[C'])\geq\chi(G[C])-qa$.
\end{itemize}
\end{lemma}

\begin{proof}
We proceed by induction on $q$. If $q=0$, we take $C'=C$: the first two
conditions are part of the hypotheses, the third is vacuous, and the
fourth holds with equality.

Now, let $q\geq1$. Since $\chi(G[C])>qa\geq(q-1)a$, the induction
hypothesis applied with $q-1$ in place of $q$ gives an induced path
$x_0-x_1-\cdots-x_{q-1}$ with $x_1,\ldots,x_{q-1}\in C$, and a set
$C''\subseteq C\setminus\{x_1,\ldots,x_{q-1}\}$ such that $G[C'']$ is
connected, $x_{q-1}$ has a neighbor in $C''$, the vertices
$x_0,\ldots,x_{q-2}$ are anticomplete to $C''$, and
$\chi(G[C''])\geq\chi(G[C])-(q-1)a$.

Set $Z=C''\setminus N_G(x_{q-1})$. Since $x_{q-1}\in C\cup\{x_0\}$ and
$C''\subseteq C$, the assumption gives
$\chi\bigl(G[N_G(x_{q-1})\cap C'']\bigr)\leq a$. As $Z$ and
$N_G(x_{q-1})\cap C''$ partition $C''$, using disjoint palettes for these
two sets yields
$
\chi(G[Z])\geq\chi(G[C''])-a\geq\chi(G[C])-qa>0.
$
Let $C'$ be the vertex set of a component of $G[Z]$ with maximum
chromatic number. Since the chromatic number of a graph is the maximum of
the chromatic numbers of its components, $G[C']$ is connected and
$\chi(G[C'])=\chi(G[Z])\geq\chi(G[C])-qa$; in particular,
$C'\neq\emptyset$.

We next choose the last vertex of the required path. The set
$N_G(x_{q-1})\cap C''$ is nonempty, by the induction hypothesis, and is
disjoint from $C'\subseteq Z$; hence $C'\subsetneq C''$. As $G[C'']$ is
connected, some vertex of $C''\setminus C'$ has a neighbor in $C'$; and
as $C'$ is a component of $G[Z]$, every such vertex lies in
$N_G(x_{q-1})$. Let $x_q$ be one, so that $x_q\in N_G(x_{q-1})\cap C''$
has a neighbor in $C'$.

It remains to verify the conclusions. The path $x_0-x_1-\cdots-x_{q-1}$
is induced and $x_q$ is adjacent to $x_{q-1}$; moreover $x_q\in C''$,
while $x_0,\ldots,x_{q-2}$ are anticomplete to $C''$. Hence
$x_0-x_1-\cdots-x_q$ is an induced path, and
$x_q\in C''\subseteq C\setminus\{x_1,\ldots,x_{q-1}\}$ gives
$x_1,\ldots,x_q\in C$. Finally,
$C'\subseteq C''\setminus N_G(x_{q-1})$. Thus, $C'$ contains none of
$x_1,\ldots,x_{q-1}$, because $C''$ contains none of them, and it does
not contain $x_q$, because $x_q\in N_G(x_{q-1})$; that is,
$C'\subseteq C\setminus\{x_1,\ldots,x_q\}$. For the same reason,
$x_{q-1}$ is anticomplete to $C'$, while $x_0,\ldots,x_{q-2}$ are
anticomplete to $C'$ because they are anticomplete to $C''$. This proves
the lemma.
\end{proof}

We next need a definition which will be used in our layering argument. A \defn{rooted configuration} in a graph $G$ is a pair $\Xi=(C,D)$ such that
\begin{itemize}
	\item $C\subseteq V(G)$; and
	\item $D=(d_1,\ldots,d_m)$ is a sequence of distinct vertices whose underlying set is an independent subset of $V(G)\setminus C$ and dominates $C$.
\end{itemize}
In set-theoretic expressions, we identify $D$ with its underlying set.
For $v\in C$, let $p_{\Xi}(v)$ be the first neighbor of $v$ in the ordering $d_1,\ldots,d_m$. For $a\in\mathbb{N}$, the configuration $\Xi$ is \defn{$a$-local} if $\chi(G[N_G(z)\cap C])\leq a$ for every $z\in C\cup D$.


For $a\in\mathbb{N}$ and $q\in\mathbb{N}\cup\{0\}$, let
\[\gamma_q(a)=\bigl(t-2+(2t-2)q\bigr)a.\]


\begin{lemma}\label{lem:rooted-pt}
	Let $a\in\mathbb{N}$ and $k\in\mathbb{N}\cup\{0\}$. Let $\Xi=(C,D)$ be an $a$-local rooted configuration in a graph $G$. If $\pi_{P_t}(G,D)\leq k$, then
	$\chi(G[C])\leq\gamma_k(a)$.
\end{lemma}

\begin{proof}
	Write $p=p_{\Xi}$. For $i\in\{1,\ldots,m\}$, let
	\[W_i=\{v\in C\mid p(v)=d_i\}.\]
	Note that $W_i\subseteq N_G(d_i)\cap C$, and $d_i$ is anticomplete
	to $W_j$ whenever $i<j$.
	
	We first observe that, if $X\subseteq C$ and $\chi(G[X])>(t-2)a$, then $G$ contains an induced path
	$d-v_1-\cdots-v_{t-1}$
	such that $d=p(v_1)\in D$ and $v_1,\ldots,v_{t-1}\in X$. Indeed, choose a component $H$ of $G[X]$ with $\chi(H)>(t-2)a$, and let $d$ be the first vertex of $D$ having a neighbor in $H$. Applying~\cref{lem:gyarfas-extension-local} with $C=V(H)$, $x_0=d$, and $q=t-2$ gives a nonempty remaining set, since its chromatic number is positive. Appending a neighbor in this set to the resulting path gives the required path. The choice of $d$ implies that $p(v_1)=d$.
	
	We prove by induction on $q$ that, for every $X\subseteq C$, if $\chi(G[X])>\gamma_q(a)$, then $G$ contains $q+1$ pairwise anticomplete induced copies of $P_t$, each of the form
	$p(v_1)-v_1-\cdots-v_{t-1}$
	with $v_1,\ldots,v_{t-1}\in X$.
	
	The case $q=0$ follows from the preceding observation. Assume the
	assertion holds for $q$, and let $X\subseteq C$ satisfy
	$\chi(G[X])>\gamma_{q+1}(a)$. For $j\in\{1,\ldots,m\}$, let
	\[X_{\geq j}=X\cap(W_j\cup\cdots\cup W_m),\qquad
	X_{>j}=X\cap(W_{j+1}\cup\cdots\cup W_m),\]
	where $X_{>m}=\emptyset$, and put $X_{\leq j}=X\setminus X_{>j}$.
	Choose $j$ maximum such that $\chi(G[X_{\geq j}])>(t-2)a$.
	Such a $j$ exists because $X_{\geq1}=X$. By maximality,
	$\chi(G[X_{>j}])\leq(t-2)a$.

	Choose a component $H$ of $G[X_{\geq j}]$ with
	$\chi(H)>(t-2)a$. Since $\chi(G[X_{>j}])\leq(t-2)a$, the component
	$H$ meets $W_j$, and $d_j$ is the first vertex of $D$ having a
	neighbor in $H$. The path construction in the preceding observation
	therefore gives an induced path $Q=d_j-v_1-\cdots-v_{t-1}$ with
	$p(v_1)=d_j$ and $v_1,\ldots,v_{t-1}\in X_{\geq j}$.

	Let $Y=X_{\leq j}\setminus N_G[V(Q)]$. Since $t\geq2$, every
	vertex of $Q$ has a neighbor in $Q$, and hence
	\[X_{\leq j}\setminus Y\subseteq
	\bigcup_{z\in V(Q)}(N_G(z)\cap C).\]
	Locality gives $\chi(G[X_{\leq j}\setminus Y])\leq ta$, so
	\[\chi(G[Y])\geq\chi(G[X])-(t-2)a-ta
	>\gamma_{q+1}(a)-(2t-2)a=\gamma_q(a).\]
	By induction, there are $q+1$ pairwise anticomplete induced copies of
	$P_t$ of the required form whose vertices in $C$ belong to $Y$.
	These vertices are anticomplete to $V(Q)$. Moreover,
	$W_j\subseteq N_G(d_j)$ implies $Y\cap W_j=\emptyset$, so every
	old root precedes $d_j$. Such a root is anticomplete to
	$V(Q)\cap C\subseteq W_j\cup\cdots\cup W_m$, and is nonadjacent
	to $d_j$ because $D$ is independent.
	Thus, the old paths together with $Q$ form $q+2$ pairwise anticomplete induced copies of $P_t$ of the required form. Taking $q=k$ and $X=C$, the inequality $\chi(G[C])>\gamma_k(a)$ would give $k+1$ pairwise anticomplete induced copies of $P_t$, each containing a vertex of $D$, contrary to $\pi_{P_t}(G,D)\leq k$. This proves the lemma.
\end{proof}

We will use the following consequence of~\cref{lem:rooted-pt}.

\begin{lemma}\label{lem:pt-dominated-set}
	Let $a,r\in\mathbb{N}$ and $k\in\mathbb{N}\cup\{0\}$, and let $G$ be a
	graph such that $\omega(G)\leq r$ and $\chi(G[N_G(v)])\leq a$ for
	every $v\in V(G)$. Let $C,D\subseteq V(G)$ be such that
	$D\cap C=\emptyset$, the set $D$ dominates $C$, and
	$\pi_{P_t}(G,D)\leq k$. Then
	$\chi(G[C])\leq\bigl(g^{\mathsf{path}}_t(r)+tka\bigr)\gamma_k(a)$.
\end{lemma}

\begin{proof}
	If $C=\emptyset$, the assertion is immediate. Thus, assume that
	$C\neq\emptyset$, and hence $D\neq\emptyset$. Let $\mathcal{Q}$ be
	a maximal collection of pairwise anticomplete induced copies of $P_t$
	in $G[D]$, and let $U$ be the union of their vertex sets. The members
	of $\mathcal{Q}$ form a $D$-rooted induced $P_t$-packing in $G$, so
	$|\mathcal{Q}|\leq k$ and $|U|\leq tk$. By maximality, the graph
	$G[D\setminus N_G[U]]$ is $P_t$-free, and hence
	$\chi(G[D\setminus N_G[U]])\leq g^{\mathsf{path}}_t(r)$ by
	Gy\'arf\'as's theorem. Since $t\geq2$, every vertex of $U$ has a
	neighbor in $U$, and hence $N_G[U]=\bigcup_{u\in U}N_G(u)$.
	Each set $N_G(u)\cap D$ is $a$-colorable. Using disjoint palettes,
	we obtain
	\[\chi(G[D])\leq g^{\mathsf{path}}_t(r)+a|U|\leq
	g^{\mathsf{path}}_t(r)+tka.\]

	Let $E_1,\ldots,E_q$ be the nonempty color classes of such a coloring
	of $G[D]$, where $q\leq g^{\mathsf{path}}_t(r)+tka$. Assign every
	vertex of $C$ to a color class containing one of its neighbors,
	thereby obtaining a partition
	$C=C_1\mathbin{\dot\cup}\cdots\mathbin{\dot\cup}C_q$ such that $E_j$
	dominates $C_j$ for every $j\in\{1,\ldots,q\}$.

	Fix $j$ with $C_j\neq\emptyset$, order $E_j$ arbitrarily, and let
	$D_j$ be the resulting sequence. Then $\Xi_j=(C_j,D_j)$ is an
	$a$-local rooted configuration in $G$: the set $E_j$ is independent, is
	disjoint from $C_j$, and dominates $C_j$, while locality follows from
	the $a$-colorability of every neighborhood in $G$. Moreover,
	$E_j\subseteq D$, and so $\pi_{P_t}(G,E_j)\leq k$. Therefore,
	\cref{lem:rooted-pt} gives $\chi(G[C_j])\leq\gamma_k(a)$. Using
	disjoint palettes for $C_1,\ldots,C_q$ proves the lemma.
\end{proof}

\begin{lemma}\label{lem:pt-tail}
	Let $a,b\in\mathbb{N}$ with $b\geq a$, and let $G$ be a connected
	graph such that $\chi(G[N_G(v)])\leq a$ for every $v\in V(G)$. If
	$\chi(G)>(t-1)a+1+2b$, then $G$ has a connected induced subgraph
	$H$ and a vertex $v_0\in V(H)$ such that the distance layering
	$L_0,\ldots,L_\ell$ of $H$ from $v_0$ satisfies the following:
	\begin{itemize}
		\item $|L_j|=1$ for $0\leq j\leq t-1$, and
		$H[L_0\cup\cdots\cup L_{t-1}]$ is an induced copy of $P_t$;
		\item $\chi(H[L_i])>b$ for some $i\geq t+1$.
	\end{itemize}
\end{lemma}

\begin{proof}
	Choose $v_0\in V(G)$, and let $C$ be the vertex set of a component
	of $G-v_0$ with maximum chromatic number. Then $v_0$ has a neighbor
	in $C$, and
	\[\chi(G[C])=\chi(G-v_0)\geq\chi(G)-1>(t-1)a+2b.\]
	Apply~\cref{lem:gyarfas-extension-local} with this set $C$, with
	$x_0=v_0$, and with $q=t-1$. We obtain an induced path
	$v_0-x_1-\cdots-x_{t-1}$ and a connected set $C'$ such that
	$x_{t-1}$ has a neighbor in $C'$, all earlier vertices of the path
	are anticomplete to $C'$, and
	\[\chi(G[C'])\geq\chi(G[C])-(t-1)a>2b.\]

	Let
	\[H=G[\{v_0,x_1,\ldots,x_{t-1}\}\cup C'],\]
	and let $L_0,\ldots,L_\ell$ be its distance layering from $v_0$.
	Since $x_{t-1}$ is the only vertex of the path with a neighbor in
	$C'$, we have $L_0=\{v_0\}$,
	$L_j=\{x_j\}$ for $1\leq j\leq t-1$, and
	$C'=L_t\cup\cdots\cup L_\ell$. Moreover,
	$L_t\subseteq N_H(x_{t-1})$, so $\chi(H[L_t])\leq a\leq b$.
	Layers of equal parity are pairwise anticomplete, and hence
	\[\chi(G[C'])\leq2\max_{i\geq t}\chi(H[L_i]).\]
	Since $\chi(G[C'])>2b$, some layer $L_i$ with $i\geq t+1$ has
	chromatic number greater than $b$.
\end{proof}

We now have everything needed to prove the following.

\begin{theorem}\label{thm:chi-bounded-paths}
	For every $t\in\mathbb{N}$, there is a function
	$f^{\mathsf{path}}_t:(\mathbb{N}\cup\{0\})\times\mathbb{N}\to\mathbb{N}$ such that,
	for all $k\in\mathbb{N}\cup\{0\}$ and $r\in\mathbb{N}$, every graph
	$G$ with $\treepi_{P_t}(G)\leq k$ and $\omega(G)\leq r$ satisfies
	$\chi(G)\leq f^{\mathsf{path}}_t(k,r)$.
\end{theorem}

\begin{proof}
If $t=1$, set
\[f^{\mathsf{path}}_1(k,r)=\Ram(k+1,r+1)-1.\]
The argument at the beginning of the subsection proves the assertion.
Thus, assume that $t\geq2$. Set
\[f^{\mathsf{path}}_t(0,r)=g^{\mathsf{path}}_t(r), \qquad
f^{\mathsf{path}}_t(k,1)=1.\]
These definitions agree at $(0,1)$, since
$g^{\mathsf{path}}_t(1)=1$. For $k\geq1$ and $r\geq2$, take
	\[a=f^{\mathsf{path}}_t(k,r-1),\qquad s=g^{\mathsf{path}}_t(r)+tka,\]
	\[M=f^{\mathsf{path}}_t(k-1,r)+s\gamma_k(a),\]
	and
	\[f^{\mathsf{path}}_t(k,r)=(t-1)a+1+2M.\]
	This recursion is well defined by induction on $k+r$.
	
	We prove the theorem by induction on $k+r$. The case $r=1$ is immediate. If $k=0$, then $G$ is $P_t$-free: an induced copy of $P_t$ would meet some bag of every tree decomposition and would therefore give a rooted packing of size one. Hence, by Gy\'arf\'as's theorem,
	$\chi(G)\leq g^{\mathsf{path}}_t(r)=f^{\mathsf{path}}_t(0,r)$.
	
	Let $k\geq1$ and $r\geq2$, and let $G$ satisfy $\treepi_{P_t}(G)\leq k$ and $\omega(G)\leq r$. Since the chromatic number of a graph is the maximum of the chromatic numbers of its components, and each component is an induced subgraph satisfying the same hypotheses, it suffices to prove the theorem when $G$ is connected.

    Take
	\[a=f^{\mathsf{path}}_t(k,r-1).\]
	For every $v\in V(G)$, the graph $G[N_G(v)]$ has clique number at most $r-1$ and induced-$P_t$-packing treewidth at most $k$. By induction,
	$\chi(G[N_G(v)])\leq a$.
	
	Suppose for a contradiction that $\chi(G)>f^{\mathsf{path}}_t(k,r)$.
	Since $s\geq1$ and $\gamma_k(a)\geq a$, we have $M\geq a$. Applying~\cref{lem:pt-tail} with
	$b=M$ gives a connected induced subgraph $H$ of $G$, a vertex $v_0\in V(H)$,
	and an index $i\geq t+1$ such that the distance layering
	$L_0,\ldots,L_\ell$ of $H$ from $v_0$ satisfies
	$\chi(H[L_i])>M$. Moreover, $|L_j|=1$ for $0\leq j\leq t-1$, and
	$H[L_0\cup\cdots\cup L_{t-1}]$ is an induced copy of $P_t$.
	Every neighborhood in $H$ is $a$-colorable.
	
	Fix a tree decomposition $\mathcal{T}=(T,\beta)$ of $H$ with
	$\pi_{P_t}(H,\mathcal{T})\leq k$.
	For $v\in V(H)$, let $T_v$ be its occurrence subtree. Let
	\[L'=L_0\cup\cdots\cup L_{i-2},\qquad T'=\bigcup_{v\in L'}T_v.\]
	Since $H[L']$ is connected, $T'$ is a subtree of $T$.
	Every vertex $z\in L'$ belongs to an induced copy of $P_t$ in
	$H[L']$: if $z\in L_j$ with $j<t-1$, use
	$H[L_0\cup\cdots\cup L_{t-1}]$, which lies in $H[L']$ because
	$i\geq t+1$; otherwise, use the first $t$ vertices of a shortest
	path from $z$ to $v_0$. Let
	\[A_i=\{v\in L_i\mid T_v\cap T'\neq\emptyset\},\qquad B_i=L_i\setminus A_i.\]
	
	\begin{claim}
		$\treepi_{P_t}(H[A_i])\leq k-1$.
	\end{claim}

	\begin{claimproof}
	For $x\in V(T')$, let $\beta_i(x)=\beta(x)\cap A_i$. Then $\mathcal{T}_i=(T',\beta_i)$ is a tree decomposition of $H[A_i]$. Indeed, the occurrence set of every $v\in A_i$ is the nonempty subtree $T_v\cap T'$. If $uv\in E(H[A_i])$, then $T_u,T_v,T'$ pairwise intersect, and hence have a common node by the Helly property for subtrees of a tree.

	Fix $x\in V(T')$, and choose $z\in\beta(x)\cap L'$, which exists by the definition of $T'$. By the preceding observation, there is an induced copy $Q_x$ of $P_t$ in $H[L']$ containing $z$, and hence meeting $\beta(x)$. Since there are no edges between $L'$ and $L_i$, the graph $Q_x$ is anticomplete to $A_i$. Consequently, a $\beta_i(x)$-rooted induced $P_t$-packing of size $k$ in $H[A_i]$, together with $Q_x$, would give a $\beta(x)$-rooted induced $P_t$-packing of size $k+1$ in $H$. Therefore, $\pi_{P_t}(H[A_i],\mathcal{T}_i)\leq k-1$, proving the claim.
	\end{claimproof}

	By induction, $\chi(H[A_i])\leq f^{\mathsf{path}}_t(k-1,r)\leq M$.
	Since $\chi(H[L_i])>M$, the set $B_i$ is nonempty.
	
	Let $C$ be a component of $H[B_i]$ with maximum chromatic number, and put
	$D=N_H(C)\cap L_{i-1}$.
	Every vertex of $C$ has a neighbor in $L_{i-1}$, so $D$ dominates $C$.
	
	Let
	\[T_C=\bigcup_{v\in C}T_v.\]
	Since $C$ is connected, $T_C$ is a subtree of $T$, and
	$T_C\cap T'=\emptyset$ because $C\subseteq B_i$. Let $P_T$ be the
	shortest path in $T$ joining $T'$ and $T_C$, and let $x$ be its
	endpoint in $T'$. For every $d\in D$, the subtree $T_d$ meets $T_C$,
	because $d$ has a neighbor in $C$, and it meets $T'$, because $d$ has
	a neighbor in $L_{i-2}$. Hence $T_d$ contains $P_T$, and in
	particular $x\in T_d$. Therefore,
	$D\subseteq\beta(x)$.
	
	Apply~\cref{lem:pt-dominated-set} in $H$. Its hypotheses hold because $\omega(H)\leq r$, every neighborhood in $H$ is $a$-colorable, the set $D$ is disjoint from $C$ and dominates $C$, and $\pi_{P_t}(H,D)\leq\pi_{P_t}(H,\beta(x))\leq k$, as $D\subseteq\beta(x)$. By the choice of $C$, we obtain
	$\chi(H[B_i])=\chi(H[C])\leq s\gamma_k(a)$.
	
	Consequently,
	$\chi(H[L_i])\leq f^{\mathsf{path}}_t(k-1,r)+s\gamma_k(a)=M$,
	contrary to the choice of $L_i$. This completes the induction, and hence finishes the proof.
\end{proof}


\subsection{Induced star-packing treewidth}
Throughout this subsection, for $d\in\mathbb{N}$, we write
$S_d=K_{1,d}$. We prove the second assertion of
\cref{thm:chi-bounded}. Since the cases
$d\in\{1,2\}$ follow from~\cref{thm:chi-bounded-paths}, it
remains to prove the following.

\begin{theorem}\label{thm:star-treepi-chi-bounded}
For every integer $d\geq3$, there is a function
$f_d^{\mathrm{star}}:(\mathbb{N}\cup\{0\})\times\mathbb{N}
\longrightarrow\mathbb{N}$ such that, for all
$k\in\mathbb{N}\cup\{0\}$ and $r\in\mathbb{N}$, every graph $G$
with $\treepi_{S_d}(G)\leq k$ and $\omega(G)\leq r$ satisfies
$\chi(G)\leq f_d^{\mathrm{star}}(k,r)$.
\end{theorem}

For the remainder of this subsection, fix an integer $d\geq3$.
For $r\in\mathbb{N}$, let $g_d^{\mathrm{star}}(r)=\Ram(d,r)$, where
$\Ram(d,r)$ is the corresponding Ramsey number. By
\cref{thm:ramsey}, every $S_d$-free graph $F$ with
$\omega(F)\leq r$ satisfies $\chi(F)\leq g_d^{\mathrm{star}}(r)$.
Indeed, for every $v\in V(F)$, the graph $F[N_F(v)]$ contains
neither a clique of size $r$ nor an independent set of size $d$. Hence
$\deg_F(v)<\Ram(d,r)$, and greedy coloring gives the required bound.

We first prove the analogue of~\cref{lem:rooted-pt} for stars.
For $a\in\mathbb{N}$ and $q\in\mathbb{N}\cup\{0\}$, let
\[\rho_{d,q}(a)=(q+1)(d-2)a^2+\bigl((d+3)q+1\bigr)a.\]

\begin{lemma}\label{lem:local-rooted-star}
Let $a\in\mathbb{N}$ and $k\in\mathbb{N}\cup\{0\}$, and let
$\Xi=(C,D)$ be an $a$-local rooted configuration in a graph
$G$. If $\pi_{S_d}(G,D)\leq k$, then
$\chi(G[C])\leq\rho_{d,k}(a)$.
\end{lemma}

\begin{proof}
Write $D=(d_1,\ldots,d_m)$ and $p=p_{\Xi}$. For every $i\in[m]$,
let \[W_i=\{v\in C\mid p(v)=d_i\}.\] Since
$W_i\subseteq N_G(d_i)\cap C$, we have $\chi(G[W_i])\leq a$.
Moreover, $d_i$ is anticomplete to $W_j$ whenever $i<j$.
Take $b=(d-2)a^2+a$. Then
$\rho_{d,q}(a)=(q+1)b+q(d+2)a$ for every
$q\in\mathbb{N}\cup\{0\}$.

We first show that, whenever $X\subseteq C$ satisfies
$\chi(G[X])>b$, there is an induced copy $Q$ of $S_d$ centered
at a vertex $v\in X$, with $p(v)$ as one leaf and all its other
vertices in $X$. Suppose otherwise. For $v\in X\cap W_i$, let
$N^+(v)=N_G(v)\cap X\cap(W_{i+1}\cup\cdots\cup W_m)$.
The set $N^+(v)$ contains no independent set of size $d-1$.
Indeed, such a set, together with $p(v)=d_i$, would be the
leaves of the required copy centered at $v$, since $d_i$ is
anticomplete to all the later sets $W_j$. Since
$G[N^+(v)]$ is $a$-colorable, it follows that
$|N^+(v)|\leq(d-2)a$.

For every $i\in[m]$, properly color $G[X\cap W_i]$ using colors
from $[a]$. For each $j\in[a]$, let $X_j$ be the set of vertices
assigned color $j$ in these colorings. Within $X_j$, each set
$X_j\cap W_i$ is independent, and every vertex has at most $(d-2)a$
neighbors in sets with larger indices. Coloring the sets
$X_j\cap W_m,\ldots,X_j\cap W_1$ in this order therefore gives
a proper coloring of $G[X_j]$ with at most $(d-2)a+1$ colors.
Using disjoint palettes for $X_1,\ldots,X_a$, we obtain
$\chi(G[X])\leq a\bigl((d-2)a+1\bigr)=b$, a contradiction.
This proves the observation.

We prove by induction on $q$ that, for every $X\subseteq C$
with $\chi(G[X])>\rho_{d,q}(a)$, there are $q+1$ pairwise
anticomplete induced copies of $S_d$, each centered at a vertex
$v\in X$, with $p(v)$ as one leaf and all its other vertices in
$X$. The case $q=0$ follows from the preceding observation,
since $\rho_{d,0}(a)=b$.

Suppose that the assertion holds for $q$, and let $X\subseteq C$
satisfy $\chi(G[X])>\rho_{d,q+1}(a)$. Take
\[c=\rho_{d,q}(a)+(d+1)a.\] For $j\in\{0,\ldots,m\}$, let
\[X_{\leq j}=X\cap(W_1\cup\cdots\cup W_j),\] where
$X_{\leq0}=\emptyset$. Choose $j$ minimum such that
$\chi(G[X_{\leq j}])>c$. Such a $j$ exists because
$X_{\leq m}=X$ and
$\rho_{d,q+1}(a)=c+a+b>c$. Set 

\[ X^-=X_{\leq j}, \quad X^+=X\setminus X^-.\] By the minimality of $j$ and the bound
$\chi(G[W_j])\leq a$, we have
$\chi(G[X^-])\leq c+a$. Consequently,
$\chi(G[X^+])>\rho_{d,q+1}(a)-(c+a)=b$.

By the observation, there is an induced copy $Q$ of $S_d$
centered at a vertex $v\in X^+$, with $p(v)$ as one leaf and
all its other vertices in $X^+$. Let
\[Y=X^-\setminus N_G[V(Q)].\] Since $V(Q)\cap X^-=\emptyset$,
we have
\[X^-\setminus Y\subseteq\bigcup_{z\in V(Q)}(N_G(z)\cap C).\]
The configuration is $a$-local, and $Q$ has $d+1$ vertices,
all belonging to $C\cup D$. Thus,
$\chi(G[X^-\setminus Y])\leq(d+1)a$, and hence
$\chi(G[Y])>c-(d+1)a=\rho_{d,q}(a)$.
By induction, there are $q+1$ pairwise anticomplete induced
copies of $S_d$ of the required form whose vertices in $C$
belong to $Y$.

Each of these copies is anticomplete to $Q$. Its vertices in
$C$ belong to $Y$, and so are anticomplete to $V(Q)$. Its leaf
in $D$ is the first neighbor of its center and therefore occurs
no later than $d_j$. On the other hand, the first neighbor of
every vertex of $Q\cap C$ occurs after $d_j$, since
$Q\cap C\subseteq X^+$. Thus, no leaf in $D$ of an old copy is
adjacent to a vertex of $Q\cap C$. Finally, these old leaves
are distinct from $p(v)$ and nonadjacent to it, because their
indices are at most $j$, the index of $p(v)$ is greater than
$j$, and $D$ is independent. Therefore, the old copies together
with $Q$ form $q+2$ pairwise anticomplete induced copies of
$S_d$ of the required form.

Taking $q=k$ and $X=C$, the inequality
$\chi(G[C])>\rho_{d,k}(a)$ would give a $D$-rooted induced
$S_d$-packing of size $k+1$, since every selected copy contains
a vertex of $D$. This contradicts
$\pi_{S_d}(G,D)\leq k$, and hence proves the lemma.
\end{proof}

We will use the following consequence of
\cref{lem:local-rooted-star}.

\begin{lemma}\label{lem:star-dominated-set}
Let $a\in\mathbb{N}$ and $k\in\mathbb{N}\cup\{0\}$, and let
$G$ be a graph such that $\chi(G[N_G(v)])\leq a$ for every
$v\in V(G)$. Let $C,D\subseteq V(G)$ be such that
	$D\cap C=\emptyset$, the set $D$ dominates $C$, and
	$\pi_{S_d}(G,D)\leq k$. Then
	$\chi(G[C])\leq\bigl((d-1)a+1+(d+1)ka\bigr)\rho_{d,k}(a)$.
\end{lemma}

\begin{proof}
If $C=\emptyset$, the assertion is immediate. Thus, assume
that $C\neq\emptyset$, and hence $D\neq\emptyset$.
Let $\mathcal{Q}$ be a maximal collection of pairwise
anticomplete induced copies of $S_d$ in $G[D]$, and let $U$ be
the union of their vertex sets. Since these copies form a
$D$-rooted induced $S_d$-packing in $G$, we have
$|\mathcal{Q}|\leq k$ and $|U|\leq(d+1)k$. By maximality,
the graph $F=G[D\setminus N_G[U]]$ is $S_d$-free.
For every $v\in V(F)$, the graph $F[N_F(v)]$ is $a$-colorable
and has no independent set of size $d$. Therefore,
$\deg_F(v)\leq(d-1)a$, and greedy coloring gives
$\chi(F)\leq(d-1)a+1$.

	Since every vertex of $U$ has a neighbor in $U$, we have
	$N_G[U]=\bigcup_{u\in U}N_G(u)$. Each set $N_G(u)\cap D$ is
	$a$-colorable. Using disjoint palettes, we obtain
	\[\chi(G[D])\leq(d-1)a+1+a|U|
	\leq(d-1)a+1+(d+1)ka.\]
Let $E_1,\ldots,E_q$ be the nonempty color classes of such a
	coloring, where \[q\leq(d-1)a+1+(d+1)ka.\]
Assign every vertex of $C$ to a class containing one of its
neighbors, obtaining pairwise disjoint sets $C_1,\ldots,C_q$
whose union is $C$, such that $E_i$ dominates $C_i$ for every
$i\in[q]$.

For each $i$ with $C_i\neq\emptyset$, order $E_i$ arbitrarily
and let $D_i$ be the resulting sequence. Then
$(C_i,D_i)$ is an $a$-local rooted configuration in $G$:
the set $E_i$ is independent, is disjoint from $C_i$,
and dominates $C_i$, while locality follows from the bound
on the chromatic number of every neighborhood in $G$.
Moreover, $E_i\subseteq D$, and so $\pi_{S_d}(G,E_i)\leq k$.
By~\cref{lem:local-rooted-star}, we have
$\chi(G[C_i])\leq\rho_{d,k}(a)$. Using disjoint palettes for
$C_1,\ldots,C_q$ proves the lemma.
\end{proof}

The following elementary observation gives us the copy of
$S_d$ required in the next proof.

\begin{lemma}\label{lem:backward-star}
Let $G$ be a connected graph, let $Q$ be an induced copy of
$S_d$ in $G$, and let $L_0=V(Q)$. For $j\geq1$, let $L_j$ be
the set of vertices at distance $j$ from $L_0$. Suppose that
every vertex of $G$ has $2d-2$ pairwise nonadjacent neighbors.
If $i\geq2$, then every vertex of $L_0\cup\cdots\cup L_{i-2}$
belongs to an induced copy of $S_d$ contained in
$G[L_0\cup\cdots\cup L_{i-2}]$.
\end{lemma}

\begin{proof}
The assertion is immediate for vertices of $L_0$. Let
$z\in L_j$, where $1\leq j\leq i-2$, choose a neighbor
$u\in L_{j-1}$, and let $Y$ be an independent set of $2d-2$ vertices
in $N_G(u)$. Every vertex of $Y$ belongs to
$L_0\cup\cdots\cup L_j$.
If $z\in Y$, then $u$, together with $z$ and any $d-1$ vertices
of $Y\setminus\{z\}$, induces $S_d$. Suppose that
$z\notin Y$. If $z$ is nonadjacent to at least $d-1$ vertices
of $Y$, then these vertices together with $z$ are the leaves
of a copy of $S_d$ centered at $u$. Otherwise, $z$ is adjacent
to at least $(2d-2)-(d-2)=d$ vertices of $Y$, which contain the
leaves of a copy of $S_d$ centered at $z$. In every case, the
resulting copy contains $z$ and is contained in
$G[L_0\cup\cdots\cup L_j]$, proving the lemma.
\end{proof}

We can now prove~\cref{thm:star-treepi-chi-bounded}.

\begin{proof}[Proof of~\cref{thm:star-treepi-chi-bounded}]
For $r\in\mathbb{N}$ and $k\in\mathbb{N}\cup\{0\}$, let
\[f_d^{\mathrm{star}}(0,r)=g_d^{\mathrm{star}}(r), \qquad f_d^{\mathrm{star}}(k,1)=1.\]
These definitions agree at $(0,1)$,
since $\Ram(d,1)=1$. For $k\geq1$ and $r\geq2$, assuming that
$f_d^{\mathrm{star}}(k,r-1)$ and
$f_d^{\mathrm{star}}(k-1,r)$ are defined, let
\[a=f_d^{\mathrm{star}}(k,r-1), \qquad s=(d-1)a+1+(d+1)ka,\]
and
\[b=f_d^{\mathrm{star}}(k-1,r)+s\rho_{d,k}(a).\]
Since $s\geq(d+1)a$ and $\rho_{d,k}(a)\geq(d+4)a$, we have
\[b\geq(d+1)(d+4)a^2.\]
In particular, $b\geq(d+1)a$ and $2b>(2d-3)a$. Set
\[f_d^{\mathrm{star}}(k,r)=2b.\]
This defines $f_d^{\mathrm{star}}$ by induction on $k+r$.

We prove the assertion by the same induction. The case $r=1$
is immediate. If $k=0$, then $G$ is $S_d$-free: an induced
copy of $S_d$ would meet a bag of every tree decomposition
and give a rooted packing of size one. Hence,
$\chi(G)\leq g_d^{\mathrm{star}}(r)=f_d^{\mathrm{star}}(0,r)$.
Thus, let $k\geq1$ and $r\geq2$, and suppose for a contradiction
that $\treepi_{S_d}(G)\leq k$, $\omega(G)\leq r$, and
$\chi(G)>f_d^{\mathrm{star}}(k,r)=2b$.

Let $H$ be an induced subgraph of $G$ with $\chi(H)=\chi(G)$,
chosen with $|V(H)|$ minimum. Then $H$ is connected, since
otherwise one of its components would have chromatic number
$\chi(H)$. Moreover, $\deg_H(v)\geq\chi(H)-1$ for every
$v\in V(H)$: otherwise, by the minimality of $H$, a proper
coloring of $H-v$ with at most $\chi(H)-1$ colors would extend
to $H$, a contradiction.
For every $v\in V(H)$, the graph $H[N_H(v)]$ has clique number
at most $r-1$ and induced $S_d$-packing treewidth at most $k$.
By induction, $\chi(H[N_H(v)])\leq a$. Since
$\deg_H(v)\geq\chi(H)-1\geq2b>(2d-3)a$, some color class in an
$a$-coloring of $H[N_H(v)]$ has at least $2d-2$ vertices.
Thus, every vertex of $H$ has $2d-2$ pairwise nonadjacent
neighbors. In particular, $H$ contains an induced copy $Q$
of $S_d$.

Let $L_0=V(Q)$ and, for $j\geq1$, let $L_j$ be the set of
vertices at distance $j$ from $L_0$ in $H$. We have
$\chi(H[L_0])=2$. Moreover,
$L_1\subseteq\bigcup_{v\in L_0}N_H(v)$, so
$\chi(H[L_1])\leq(d+1)a$. Since edges join only equal or
consecutive layers (see~\cref{sec:prelim}), we have
$\chi(H)\leq2\max_{j\geq0}\chi(H[L_j])$.
As $\chi(H)>2b$, some layer has chromatic number greater
than $b$. Since $b\geq(d+1)a\geq2$, this layer is $L_i$ for
some $i\geq2$.

Now, fix a tree decomposition $\mathcal{T}=(T,\beta)$ of $H$ with
$\pi_{S_d}(H,\mathcal{T})\leq k$, and let $T_v$ be the
occurrence subtree of $v$ for every $v\in V(H)$. Let
\[L'=L_0\cup\cdots\cup L_{i-2},\] and
\[T'=\bigcup_{v\in L'}T_v.\] Since $H[L']$ is connected, $T'$
is a subtree of $T$. Let
\[A_i=\{v\in L_i\mid T_v\cap T'\neq\emptyset\}, \qquad B_i=L_i\setminus A_i.\]

\begin{claim}
    $\treepi_{S_d}(H[A_i])\leq k-1$.
\end{claim}

\begin{claimproof}
For $x\in V(T')$, let $\beta_i(x)=\beta(x)\cap A_i$.
Then $\mathcal{T}_i=(T',\beta_i)$ is a tree decomposition of
$H[A_i]$. Indeed, the occurrence set of every $v\in A_i$ is
the nonempty subtree $T_v\cap T'$. If $uv\in E(H[A_i])$,
then $T_u,T_v,T'$ pairwise intersect, and therefore have a
common node by the Helly property for subtrees of a tree.

Fix $x\in V(T')$ and choose $z\in\beta(x)\cap L'$, which
exists by the definition of $T'$. By~\cref{lem:backward-star},
there is an induced copy $Q_x$ of $S_d$ in $H[L']$ containing
$z$. Since $L'$ is anticomplete to $L_i$, a
$\beta_i(x)$-rooted induced $S_d$-packing of size $k$ in
$H[A_i]$, together with $Q_x$, would give a
$\beta(x)$-rooted induced $S_d$-packing of size $k+1$ in $H$.
Consequently, $\pi_{S_d}(H[A_i],\mathcal{T}_i)\leq k-1$,
proving the claim.
\end{claimproof}

By induction,
$\chi(H[A_i])\leq f_d^{\mathrm{star}}(k-1,r)\leq b$.
Since $\chi(H[L_i])>b$, the set $B_i$ is nonempty. Let $C$ be the vertex set of a component
of $H[B_i]$ with maximum chromatic number, and take
$D=N_H(C)\cap L_{i-1}$. Every vertex of $C$ has a neighbor
in $L_{i-1}$, so $D$ dominates $C$.

Let $T_C=\bigcup_{v\in C}T_v$. Since $H[C]$ is connected,
$T_C$ is a subtree of $T$, and $T_C\cap T'=\emptyset$
because $C\subseteq B_i$. Let $P$ be the shortest path in
$T$ joining $T'$ and $T_C$, and let $x$ be its endpoint in
$T'$. Every $z\in D$ has a neighbor in $C$ and a neighbor in
$L_{i-2}$. Hence, $T_z$ meets both $T_C$ and $T'$ and
therefore contains $P$. In particular, $x\in T_z$ for every
$z\in D$, and so $D\subseteq\beta(x)$.

Apply~\cref{lem:star-dominated-set} in $H$. Its hypotheses
hold because $D$ is disjoint from $C$ and dominates $C$,
every neighborhood in $H$ is $a$-colorable, and
$\pi_{S_d}(H,D)\leq\pi_{S_d}(H,\beta(x))\leq k$, as
$D\subseteq\beta(x)$. We obtain
$\chi(H[C])\leq s\rho_{d,k}(a)$.
By the choice of $C$, this gives
$\chi(H[B_i])\leq s\rho_{d,k}(a)$. Consequently,
$\chi(H[L_i])\leq f_d^{\mathrm{star}}(k-1,r)
+s\rho_{d,k}(a)=b$, contrary to the choice of $i$.
This completes the proof of the theorem.
\end{proof}

Together with~\cref{thm:chi-bounded-paths},~\cref{thm:star-treepi-chi-bounded} completes the
proof of~\cref{thm:chi-bounded}.

\medskip

Let us conclude this subsection by remarking that the smallest tree that is
neither a path nor a star is the $S_{2,1,1}$, which is the graph obtained
from $K_{1,3}$ by subdividing one edge once (also known as \textit{chair} or
\textit{fork}). This leads to the following
question:

\begin{question}
    Is it true that graphs of bounded induced $S_{2,1,1}$-packing
treewidth form a $\chi$-bounded class?
\end{question}

\subsection{The Erdős--Hajnal property}
We finish the section with the proof of~\cref{thm:EH}. We use the
following repeated-substitution consequence of a theorem of Alon, Pach,
and Solymosi~\cite[Theorem~1.1]{AlonPachSolymosi}. Let $F$ be a graph
with $V(F)=\{v_1,\ldots,v_h\}$, and let $F_1,\ldots,F_h$ be graphs.
We write $F(F_1,\ldots,F_h)$ for the graph obtained from vertex-disjoint
copies of $F_1,\ldots,F_h$ by making the copy of $F_i$ complete to the
copy of $F_j$ precisely when $v_iv_j\in E(F)$.

\begin{lemma}[Alon, Pach, and Solymosi~\cite{AlonPachSolymosi}]
\label{lem:APS-EH}
If, for $H \in \{F,F_1,\ldots,F_h\}$, $H$-free graphs have the Erdős--Hajnal property, then so
have $F(F_1,\ldots,F_h)$-free graphs.
\end{lemma}

Let us also record the following simple property of tree decompositions.

\begin{lemma}\label{lem:tree-decomposition-eh}
Let $G$ be a nonempty graph with a tree decomposition $(T,\beta)$, and
let $\delta>0$. If
\[
\max\{\alpha(G[\beta(x)]),\omega(G[\beta(x)])\}
\geq |\beta(x)|^\delta
\qquad\text{for every }x\in V(T),
\]
then
\[
\max\{\alpha(G),\omega(G)\}\geq |V(G)|^{\delta/(\delta+1)}.
\]
\end{lemma}

\begin{proof}
Set
\[
m=\max\{\alpha(G),\omega(G)\},
\qquad
b=\max_{x\in V(T)}|\beta(x)|.
\]
Applying the hypothesis to a bag of size $b$ gives $b^\delta\leq m$.
Also, $\tw(G)\leq b-1$, so $\chi(G)\leq b$. Consequently,
\[
|V(G)|\leq\alpha(G)\chi(G)\leq mb\leq m^{1+1/\delta},
\]
which implies the result.
\end{proof}

\EH*

\begin{proof}
If $\mathcal{G}_{H,k}$ has the Erdős--Hajnal property, then the class
of $H$-free graphs does too, because every $H$-free graph has
$\treepi_H(G)=0$. Conversely, suppose that $H$-free graphs have the
Erdős--Hajnal property.

Let $q=k+1$, let $I_q$ denote the edgeless graph on $q$ vertices, and
let $qH$ denote the disjoint union of $q$ copies of $H$. Since
$qH=I_q(H,\ldots,H)$,~\cref{lem:APS-EH} shows that $qH$-free graphs have the
Erdős--Hajnal property: the property holds for $H$ by assumption and
for $I_q$ by Ramsey's theorem. Hence, there is a constant $\delta>0$
such that every $qH$-free graph $J$ satisfies
\[
\max\{\alpha(J),\omega(J)\}\geq |V(J)|^\delta.
\]

Let $G\in\mathcal{G}_{H,k}$. If $G$ is empty, there is nothing to
prove, so fix a tree decomposition $\mathcal{T}=(T,\beta)$ of $G$
with $\pi_H(G,\mathcal{T})\leq k$. Every bag $G[\beta(x)]$ is
$qH$-free: otherwise, it contains $q=k+1$ pairwise anticomplete
induced copies of $H$, which form a $\beta(x)$-rooted induced-$H$
packing in $G$, contrary to $\pi_H(G,\mathcal{T})\leq k$. Therefore,
\[
\max\{\alpha(G[\beta(x)]),\omega(G[\beta(x)])\}
\geq |\beta(x)|^\delta
\qquad\text{for every }x\in V(T).
\]
Applying~\cref{lem:tree-decomposition-eh} gives
\[
\max\{\alpha(G),\omega(G)\}\geq |V(G)|^{\delta/(\delta+1)}.
\]
Since $\delta$ depends only on $H$ and $k$, this proves that
$\mathcal{G}_{H,k}$ has the Erdős--Hajnal property.
\end{proof}

\section{Sim-width and induced packing treewidth}\label{sec:sim}
In this section, we characterize the relationship between sim-width and induced $H$-packing treewidth. We first introduce the rooted $H$-obstructions and establish that they inherently carry large induced $H$-packing treewidth.
We then show that sim-width and induced $P_3$-packing treewidth are strictly incomparable in general.
Building on this structural barrier, we finally prove \cref{thm:simw-rooted-H-obstruction-restate}, which provides a unified framework that simultaneously generalizes and resolves \cref{question:abrishami-etal,question:Brettell-etal,question:Storgel-etal}.

\hobstruction*

\subsection{Branch decompositions and sim-width}
We start with a formal introduction of branch decompositions and sim-width.
Branch decompositions provide a common framework for several width
parameters that measure the complexity of cuts in a graph. Let $G$ be a
graph with at least two vertices. A \defn{branch decomposition} of $G$ is a
pair $\mathcal D=(T,\delta)$, where $T$ is a subcubic tree and $\delta$ is a
bijection from $V(G)$ to the set of leaves of $T$. Every edge $e\in E(T)$
displays a bipartition $(A_e,\overline A_e)$ of $V(G)$: the two components
of $T-e$ contain respectively the leaves in $\delta(A_e)$ and
$\delta(\overline A_e)$, where
$\overline A_e=V(G)\setminus A_e$. The names of the two shores may be
interchanged.

For disjoint sets $A,B\subseteq V(G)$, we denote by $G[A,B]$ the bipartite
graph with bipartition $(A,B)$ whose edges are precisely the edges of $G$
with one endpoint in $A$ and the other in $B$. Thus,
$G[A,\overline A]$ records exactly the edges of $G$ crossing the cut
$(A,\overline A)$.

We use the standard definition of sim-width due to Kang, Kwon, Str{\o}mme,
and Telle~\cite{KangKwonStrommeTelle2017}. For $A\subseteq V(G)$, let
$\overline A=V(G)\setminus A$. We denote by
$\mathsf{cutsim}_G(A,\overline A)$ the maximum size of a matching
\[
    M=\{x_1y_1,\ldots,x_my_m\}
\]
such that $x_i\in A$ and $y_i\in\overline A$ for every $i\in[m]$, and the
endpoint set
\[
    \{x_1,\ldots,x_m,y_1,\ldots,y_m\}
\]
induces exactly the edges $x_1y_1,\ldots,x_my_m$ in $G$. Equivalently, the
selected endpoints in $A$ are independent, the selected endpoints in
$\overline A$ are independent, and no off-diagonal edge $x_iy_j$ with
$i\ne j$ is present. Such a matching is called a \defn{sim-matching} across
the cut $(A,\overline A)$.

The \defn{sim-width} of a branch decomposition
$\mathcal D=(T,\delta)$ of $G$ is
\[
    \simw_G(\mathcal D)
    =
    \max_{e\in E(T)}
    \mathsf{cutsim}_G(A_e,\overline A_e),
\]
and the \defn{sim-width} of $G$ is
\[
    \simw(G)
    =
    \min_{\mathcal D}\simw_G(\mathcal D),
\]
where the minimum is taken over all branch decompositions of $G$. For
graphs on at most one vertex, set $\simw(G)=0$.

\subsection{Rooted obstructions}

We start by defining the rooted $H$-obstructions. Let $H$ be a graph, let $u,v\in V(H)$, and let $r\in\mathbb N$. An \defn{$(H,u,v)$-obstruction of order $r$} is a graph $F$ for which there is a partition
\[
    V(F)=X_1\mathbin{\dot\cup}\cdots\mathbin{\dot\cup}X_r
    \mathbin{\dot\cup}
    Y_1\mathbin{\dot\cup}\cdots\mathbin{\dot\cup}Y_r
\]
and isomorphisms
\[
    \varphi_i:H\to F[X_i]
    \qquad\text{and}\qquad
    \psi_i:H\to F[Y_i]
    \qquad\text{for every }i\in[r],
\]
such that $X_i$ is anticomplete to $X_j$ and $Y_i$ is anticomplete to $Y_j$ for all distinct $i,j\in[r]$, and
$\varphi_i(u)\psi_j(v)\in E(F)$ for all $i,j\in[r]$.
No condition is imposed on the remaining edges between $\bigcup_{i\in[r]}X_i$ and $\bigcup_{j\in[r]}Y_j$. A \defn{rooted $H$-obstruction of order $r$} is an $(H,u,v)$-obstruction of order $r$ for some $u,v\in V(H)$. See Figure~\ref{fig:rooted-obstruction} for an example.

When $H=K_1$, rooted $H$-obstructions of order $r$ are precisely induced copies of $K_{r,r}$. When $H=P_2$, they are precisely the $r$-obstructions from~\cref{question:abrishami-etal}.

\begin{figure}[t]
	\centering
	\includegraphics[width=0.55\linewidth]{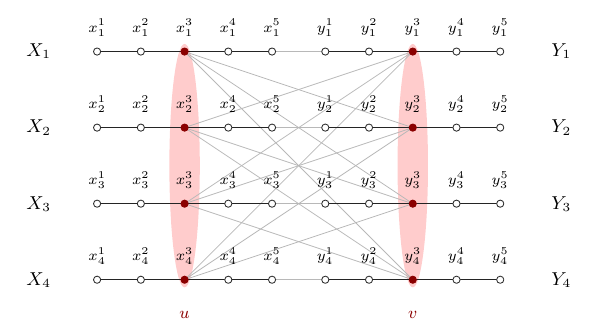}
    \caption{A rooted $P_5$-obstruction of order $r=4$.}
    \label{fig:rooted-obstruction}
\end{figure}

The following lemma explains the terminology.

\begin{lemma}\label{lem:rooted-H-obstruction-lower-bound}
Let $H$ be a graph and let $F$ be a rooted $H$-obstruction of order $r$. Then
$
    \treepi_H(F)\geq r.
$
\end{lemma}

\begin{proof}
Let $F$ be an $(H,u,v)$-obstruction of order $r$, witnessed by $X_1,\ldots,X_r,Y_1,\ldots,Y_r$ and isomorphisms $\varphi_i,\psi_i$. Let
\[
    x_i=\varphi_i(u)
    \qquad\text{and}\qquad
    y_i=\psi_i(v)
    \qquad\text{for every }i\in[r].
\]
Let $\mathcal T=(T,\beta)$ be a tree decomposition of $F$, and for $z\in V(F)$ let
\[
    T_z=\{t\in V(T)\mid z\in\beta(t)\}.
\]
Since $x_iy_j\in E(F)$ for all $i,j\in[r]$, we have
$T_{x_i}\cap T_{y_j}\ne\emptyset$ for all $i,j\in[r]$. If the
subtrees $T_{x_1},\ldots,T_{x_r}$ are pairwise intersecting, then
they have a common node by the Helly property. Otherwise, choose
distinct $i,i'\in[r]$ such that
$T_{x_i}\cap T_{x_{i'}}=\emptyset$, and let $Q$ be the unique path
of $T$ joining these two subtrees. Every subtree $T_{y_j}$ intersects
both $T_{x_i}$ and $T_{x_{i'}}$, and therefore contains $Q$.
Thus, some bag of $\mathcal T$ intersects every member of $\{X_i\mid i\in[r]\}$ or every member of $\{Y_j\mid j\in[r]\}$. Both families are induced $H$-packings of size $r$. Since $\mathcal T$ was arbitrary, $\treepi_H(F)\geq r$. This proves~\cref{lem:rooted-H-obstruction-lower-bound}.
\end{proof}

\subsection{Incomparability with sim-width}

For every integer $n\geq 1$, define a graph $G_n$ as follows; see Figure~\ref{fig:Gn-construction}. Start with a complete bipartite graph with bipartition $A_n=\{a_1,\ldots,a_n\}$ and $B_n=\{b_1,\ldots,b_n\}$. For every vertex $v\in A_n\cup B_n$, add two private leaves $v'$ and $v''$, each adjacent only to $v$. We call $C_n=A_n\cup B_n$ the \defn{core} of $G_n$.

\begin{proposition}\label{prop:teepi-versus-simwidth}
The following statements hold.
\begin{enumerate}[\rm (i)]
    \item There is a graph class $\mathcal G_1$ such that $\treepi_{P_3}(G)\leq 2$ for every $G\in\mathcal G_1$, but $\mathcal G_1$ has unbounded sim-width.\label{clm:sim-i}

    \item There is a graph class $\mathcal G_2$ such that $\simw(G)\leq 1$ for every $G\in\mathcal G_2$, but $\{\treepi_{P_3}(G)\mid G\in\mathcal G_2\}$ is unbounded.\label{clm:sim-ii}
\end{enumerate}
Consequently, the parameters $\treepi_{P_3}$ and sim-width are incomparable.
\end{proposition}

\begin{figure}[t]
    \centering
    \includegraphics[width=0.37\linewidth]{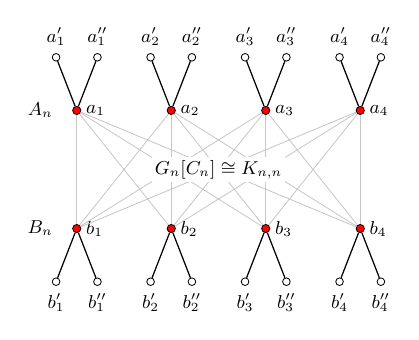}
    \caption{The graph $G_n$, illustrated for $n=4$. The core $C_n=A_n\cup B_n$ induces a complete bipartite graph $K_{n,n}$. Each core vertex $v\in C_n$ has two private leaves $v'$ and $v''$.}
    \label{fig:Gn-construction}
\end{figure}

\begin{proof}
We first prove~\ref{clm:sim-i}. Let $\mathcal G_1$ be the class of $(K_6,3P_3)$-free graphs. Munaro and Yang~\cite{MunaroYang2023}, using a result of Kang, Kwon, Strømme and Telle~\cite{KangKwonStrommeTelle2017}, showed that, for every $u\geq 3$, the class of $(K_6,uP_3)$-free graphs has unbounded sim-width. Hence $\mathcal G_1$ has unbounded sim-width.
If $G\in\mathcal G_1$, then $G$ has no induced $P_3$-packing of size three, and therefore $\pi_{P_3}(G,V(G))\leq 2$. The one-bag tree decomposition of $G$ now gives $\treepi_{P_3}(G)\leq 2$. This proves~\ref{clm:sim-i}.

We now prove~\ref{clm:sim-ii}. Let $\mathcal G_2=\{G_n\mid n\geq 1\}$. Fix $n\geq 1$. Choose a branch decomposition
$\mathcal D^0=(R^0,\delta^0)$ of $G_n[C_n]$ such that no two leaves
of $R^0$ are adjacent. When $n=1$, take $R^0$ to be the path on
three vertices, with its two leaves corresponding to $a_1$ and
$b_1$. For every $v\in C_n$, let $\ell_v^0=\delta^0(v)$ and let $p_v$ be its neighbor in $R^0$. Delete $\ell_v^0$, add adjacent vertices $x_v,y_v$ and leaves $\ell_v,\ell_{v'},\ell_{v''}$, and add the edges
\[
    p_vx_v,\quad x_v\ell_v,\quad x_vy_v,\quad
    y_v\ell_{v'},\quad y_v\ell_{v''}.
\]
Set $\delta(v)=\ell_v$, $\delta(v')=\ell_{v'}$, and $\delta(v'')=\ell_{v''}$. Performing this operation for every $v\in C_n$ gives a subcubic tree $R$ whose leaves are precisely the vertices in the image of $\delta$. Thus, $\mathcal D=(R,\delta)$ is a branch decomposition of $G_n$.

We claim that every cut displayed by $\mathcal D$ has cut-sim value at most one. For a cut displayed by one of the new edges associated with a vertex $v\in C_n$, either every crossing edge is incident with $v$, or the crossing edges are $vv'$ and $vv''$. Hence such a cut admits no sim-matching of size two.
For every other displayed cut, each private leaf lies on the same side as its neighbor in the core. Thus, only edges of the complete bipartite graph $G_n[C_n]$ cross the cut. Suppose that two crossing core edges form a sim-matching. If they have the same orientation across the cut, then the corresponding off-diagonal core edges are present. If they have opposite orientations, then two selected endpoints on one side of the cut are adjacent. Both cases contradict the definition of a sim-matching. Therefore $\simw_{G_n}(\mathcal D)\leq 1$, and hence $\simw(G_n)\leq 1$.

For each $i\in[n]$, let
\[
    X_i=\{a_i',a_i,a_i''\}
    \qquad\text{and}\qquad
    Y_i=\{b_i',b_i,b_i''\}.
\]
The graph $G_n$ is a rooted $P_3$-obstruction of order $n$, with roots $a_i$ in $X_i$ and $b_i$ in $Y_i$. By~\cref{lem:rooted-H-obstruction-lower-bound}, $\treepi_{P_3}(G_n)\geq n$. This proves~\ref{clm:sim-ii}, and hence the proposition.
\end{proof}

\subsection{Sim-width of blob graphs}

We next establish the sim-width ingredient for the converse direction,
using blob graphs introduced and studied by~\cite{gartland2021finding}.
The \defn{blob graph} $G^\circ$ of a graph $G$ has as its vertices the nonempty connected subsets of $V(G)$. Two distinct vertices $X,Y\in V(G^\circ)$ are adjacent if $X\cap Y\ne\emptyset$ or there is an edge of $G$ between $X$ and $Y$.

Let $\cH\ne\emptyset$ be a fixed finite family of connected graphs. For a graph $G$ and a set $B\subseteq V(G)$, let
\begin{align*}
    G^\circ_{\cH}
    &=G^\circ[\{X\subseteq V(G)\mid G[X]\text{ is isomorphic to a member of }\cH\}],\\
    G^\circ_{\cH}(B)
    &=G^\circ_{\cH}[\{X\in V(G^\circ_{\cH})\mid X\cap B\ne\emptyset\}].
\end{align*}
Thus,
\[
    \pi_{\cH}(G,B)=\alpha(G^\circ_{\cH}(B)).
\]
For a connected graph $H$, we write $G_H^\circ=G_{\{H\}}^\circ$.

We first need the following result from our earlier work~\cite[Lemma~9.1]{nikabadi2026induced}.

\begin{restatable}{lemma}{blobequivalence}\label{lem:blob-equivalence}
Let $\cH\neq\emptyset$ be a fixed finite family of connected graphs. For every graph $G$,
$\treepi_{\cH}(G)=\treealpha(G^\circ_{\cH})$.
Furthermore, given a tree decomposition of either graph, a corresponding tree decomposition of the other can be computed in time polynomial in $|V(G)|$, $|V(G^\circ_\cH)|$, and the size of the given decomposition.
\end{restatable}

For a graph $G$, let $\cH_G$ be the set of all subgraphs of $G$
isomorphic to a member of $\cH$. The \defn{$\cH$-graph} of $G$\footnote{The $\cH$-graph of $G$ is generally different from
$G^\circ_{\cH}$. For example, if $\cH=\{P_3\}$ and $G=K_3$, then
$G^\circ_{\cH}$ has no vertices, since $K_3$ has no induced copy of
$P_3$, whereas the $\cH$-graph of $G$ has three vertices,
corresponding to the three spanning $P_3$-subgraphs of $K_3$, one for
each choice of the omitted edge.},
denoted by $\cH(G)$, has vertex set $\cH_G$, and two distinct members
$F_1,F_2\in\cH_G$ are adjacent if $V(F_1)\cap V(F_2)\neq\emptyset$ or
there is an edge of $G$ with one endpoint in $V(F_1)$ and the other in
$V(F_2)$. We use the following result of Munaro and Yang.

\begin{theorem}[Munaro and Yang~\cite{MunaroYang2023}]
\label{thm:munaro-yang-H-graph}
Let $\cH\neq\emptyset$ be a fixed finite family of connected nonnull
graphs, and let
$
    h=\max_{J\in\cH}|V(J)|.
$
Let $G$ be a graph, and let $\mathcal D=(T,\delta)$ be a branch
decomposition of $G$. If $|V(\cH(G))|>1$, then one can construct, in
time
$
    O\bigl(|V(T)|+|V(G)|^{h+1}\bigr),
$
a branch decomposition
$
    \widehat{\mathcal D}
    =
    (\widehat T,\widehat\delta)
$
of $\cH(G)$ such that
$
    \simw_{\cH(G)}(\widehat{\mathcal D})
    \leq
    \simw_G(\mathcal D).
$
\end{theorem}

The graph $G^\circ_{\cH}$ is naturally an induced subgraph of
$\cH(G)$: it is obtained by retaining precisely those vertices of
$\cH(G)$ that are induced subgraphs of $G$. We need the following
customization of~\cref{thm:munaro-yang-H-graph}.

\begin{lemma}\label{lem:munaro-yang-simw}
Let $\cH\neq\emptyset$ be a fixed finite family of connected nonnull
graphs, and let
$
    h=\max_{J\in\cH}|V(J)|.
$
Let $G$ be a graph, and let $\mathcal D=(T,\delta)$ be a branch
decomposition of $G$. If $|V(G^\circ_{\cH})|>1$, then one can
construct, in time
$
    O\bigl(|V(T)|+|V(G)|^{h+1}\bigr),
$
a branch decomposition $\mathcal D'=(T',\delta')$ of
$G^\circ_{\cH}$ such that
$
    \simw_{G^\circ_{\cH}}(\mathcal D')
    \leq
    \simw_G(\mathcal D).
$
Consequently, for every graph $G$,
\begin{equation}\label{eq:simw-induced-H-graph}
    \simw(G^\circ_{\cH})\leq\simw(G).
\end{equation}
\end{lemma}

\begin{proof}
Let
$
    n=|V(G)|
$
and let
\[
    \mathcal I
    =
    \{F\in V(\cH(G))\mid F=G[V(F)]\}.
\]
Thus, $\mathcal I$ consists precisely of the members of $\cH_G$ that
are induced subgraphs of $G$.

Define
\[
    \iota:V(G^\circ_{\cH})\to\mathcal I
\]
by
$
    \iota(X)=G[X]
$
for every $X\in V(G^\circ_{\cH})$. We first observe that $\iota$ is
an isomorphism from $G^\circ_{\cH}$ to the induced subgraph
$\cH(G)[\mathcal I]$. Indeed, by the definition of
$G^\circ_{\cH}$, the graph $G[X]$ is isomorphic to a member of $\cH$
for every $X\in V(G^\circ_{\cH})$, and hence
$\iota(X)\in\mathcal I$. Conversely, if $F\in\mathcal I$, then
$
    F=G[V(F)]
$
is isomorphic to a member of $\cH$, and therefore
$
    V(F)\in V(G^\circ_{\cH}).
$
Thus, $\iota$ is a bijection.

Let $X,Y\in V(G^\circ_{\cH})$ be distinct. By the definition of the
blob graph, $X$ and $Y$ are adjacent in $G^\circ_{\cH}$ if and only
if
$
    X\cap Y\neq\emptyset
$
or there is an edge of $G$ with one endpoint in $X$ and the other in
$Y$. By the definition of $\cH(G)$, this is equivalent to
$\iota(X)$ and $\iota(Y)$ being adjacent in $\cH(G)$. It follows that
$\iota$ is an isomorphism from $G^\circ_{\cH}$ to
$\cH(G)[\mathcal I]$.

Since
$
    |V(G^\circ_{\cH})|>1,
$
we have
$
    |V(\cH(G))|>1.
$
Apply~\cref{thm:munaro-yang-H-graph} to $G$ and $\mathcal D$. We
obtain, in time
$
    O\bigl(|V(T)|+n^{h+1}\bigr),
$
a branch decomposition
$
    \widehat{\mathcal D}
    =
    (\widehat T,\widehat\delta)
$
of $\cH(G)$ such that
$
    \simw_{\cH(G)}(\widehat{\mathcal D})
    \leq
    \simw_G(\mathcal D).
$

Let $T'$ be the minimal subtree of $\widehat T$ containing the set of
leaves
$
    \widehat\delta(\mathcal I)
    =
    \{\widehat\delta(F)\mid F\in\mathcal I\}.
$
Define
\[
    \delta':V(G^\circ_{\cH})\to V(T')
\]
by
$
    \delta'(X)=\widehat\delta(\iota(X))
$
for every $X\in V(G^\circ_{\cH})$.

We claim that $\mathcal D'=(T',\delta')$ is a branch decomposition of
$G^\circ_{\cH}$. Since $T'$ is a subtree of the subcubic tree
$\widehat T$, the tree $T'$ is subcubic. Moreover, every vertex of
$\widehat\delta(\mathcal I)$ is a leaf of $\widehat T$. Since
$|\mathcal I|>1$, each of these vertices is also a leaf of $T'$.
Conversely, every leaf of $T'$ belongs to
$\widehat\delta(\mathcal I)$: otherwise, deleting such a leaf from
$T'$ would leave a smaller subtree of $\widehat T$ containing all
vertices of $\widehat\delta(\mathcal I)$, contrary to the minimality
of $T'$. Thus, the leaves of $T'$ are precisely the vertices of
$\widehat\delta(\mathcal I)$. Since both $\iota$ and
$\widehat\delta|_{\mathcal I}$ are bijections, $\delta'$ is a
bijection from $V(G^\circ_{\cH})$ to the set of leaves of $T'$. This
proves the claim.

It remains to bound the sim-width of $\mathcal D'$. Fix an edge
$e\in E(T')$. Let
$
    (\widehat A_e,\overline{\widehat A_e})
$
be the bipartition of $V(\cH(G))$ displayed by $e$ in
$\widehat{\mathcal D}$, and let
$
    (A_e,\overline A_e)
$
be the bipartition of $V(G^\circ_{\cH})$ displayed by $e$ in
$\mathcal D'$. After possibly interchanging the names of the shores,
we have
\[
    \iota(A_e)=\widehat A_e\cap\mathcal I
    \qquad\text{and}\qquad
    \iota(\overline A_e)
    =
    \overline{\widehat A_e}\cap\mathcal I.
\]

Let
$
    M=\{X_1Y_1,\ldots,X_mY_m\}
$
be a sim-matching in $G^\circ_{\cH}$ across
$(A_e,\overline A_e)$. Since $\iota$ is an isomorphism from
$G^\circ_{\cH}$ to the induced subgraph $\cH(G)[\mathcal I]$, the
edges
$
    \iota(X_1)\iota(Y_1),\ldots,
    \iota(X_m)\iota(Y_m)
$
form a matching in $\cH(G)$ across
$(\widehat A_e,\overline{\widehat A_e})$. Moreover, their endpoint
set induces exactly these matching edges in
$\cH(G)[\mathcal I]$. Since $\cH(G)[\mathcal I]$ is an induced
subgraph of $\cH(G)$, the same endpoint set induces exactly these
edges in $\cH(G)$. Hence these edges form a sim-matching in
$\cH(G)$ across
$(\widehat A_e,\overline{\widehat A_e})$. Therefore,
$
    \mathsf{cutsim}_{G^\circ_{\cH}}(A_e,\overline A_e)
    \leq
    \mathsf{cutsim}_{\cH(G)}
    (\widehat A_e,\overline{\widehat A_e})
    \leq
    \simw_{\cH(G)}(\widehat{\mathcal D})
    \leq
    \simw_G(\mathcal D).
$
Since this holds for every edge $e\in E(T')$, we obtain
$
    \simw_{G^\circ_{\cH}}(\mathcal D')
    \leq
    \simw_G(\mathcal D).
$

We finally verify the running time. Since $\cH$ is fixed and every
member of $\cH$ has at most $h$ vertices, the graph $\cH(G)$ has
$
    O(n^h)
$
vertices. Indeed, there are $O(n^j)$ choices for the vertex set of a
$j$-vertex subgraph, where $j\leq h$, and, for each such vertex set,
only a constant number of possible subgraphs. For each
$F\in V(\cH(G))$, whether
$
    F=G[V(F)]
$
can be checked in time $O(h^2)$. Thus, the set $\mathcal I$ and the
isomorphism $\iota$ can be computed in time $O(n^h)$.

The minimal subtree $T'$ can be obtained from $\widehat T$ by
repeatedly deleting leaves that do not belong to
$\widehat\delta(\mathcal I)$. This takes time linear in the size of
$\widehat T$. Since $\widehat{\mathcal D}$ is constructed in time
$O\bigl(|V(T)|+n^{h+1}\bigr)$, its size, and hence the time required
for this restriction, is $O\bigl(|V(T)|+n^{h+1}\bigr)$. Therefore,
the entire construction of $\mathcal D'$ takes time
$
    O\bigl(|V(T)|+n^{h+1}\bigr).
$

To prove~\eqref{eq:simw-induced-H-graph}, first suppose that
$
    |V(G^\circ_{\cH})|>1.
$
Choose a branch decomposition $\mathcal D$ of $G$ satisfying
$
    \simw_G(\mathcal D)=\simw(G).
$
The first part of the lemma gives a branch decomposition
$\mathcal D'$ of $G^\circ_{\cH}$ such that
\[
    \simw(G^\circ_{\cH})
    \leq
    \simw_{G^\circ_{\cH}}(\mathcal D')
    \leq
    \simw_G(\mathcal D)
    =
    \simw(G).
\]
If
$
    |V(G^\circ_{\cH})|\leq1,
$
then $\simw(G^\circ_{\cH})=0$, and the same inequality
holds. This proves~\eqref{eq:simw-induced-H-graph}, and hence finishes the proof of~\cref{lem:munaro-yang-simw}.
\end{proof}


\subsection{Separators and rooted obstructions}

We next extract from~\cref{lem:hit-vs-anti} a consequence for cuts of
bounded cut-sim value. Let $c_{\ref{lem:hit-vs-anti}}(\cdot,\cdot,\cdot)$ denote the function from that lemma.

\begin{lemma}\label{lem:sim-cut-hitting-set}
Let $a,q\in\mathbb N$, let $G$ be a $K_{a,a}$-free graph, and let $A\subseteq V(G)$. If
$
    \mathsf{cutsim}_G(A,\overline A)<q,
$
then there is a set $X\subseteq V(G)$ such that
$
    \alpha(G[X])<c_{\ref{lem:hit-vs-anti}}(a,2,q)
$
and $X\cap\{x,y\}\ne\emptyset$ for every $x\in A$ and $y\in\overline A$ with $xy\in E(G)$.
\end{lemma}

\begin{proof}
Apply~\cref{lem:hit-vs-anti} to the $2$-system
$
    \mathcal S_A=\{\{x,y\}\mid x\in A,\ y\in\overline A,\ xy\in E(G)\}.
$
If $\mathcal S_A$ contains an anticomplete subfamily of size $q$, then the corresponding crossing edges form a sim-matching of size $q$ across $(A,\overline A)$, a contradiction. Hence $\mathcal S_A$ has a hitting set $X$ with $\alpha(G[X])<c_{\ref{lem:hit-vs-anti}}(a,2,q)$.
\end{proof}

The next lemma contains the balanced-separator argument needed for the main theorem.

\begin{lemma}\label{lem:simw-balanced-separator}
Let $a\in\mathbb N$ and $s\in\mathbb N\cup\{0\}$, and let $G$ be a $K_{a,a}$-free graph with $\simw(G)\leq s$. For every vertex-weighting $\wei$ of $G$, there is a balanced separator $Z$ of $(G,\wei)$ such that
$
    \alpha(G[Z])<2c_{\ref{lem:hit-vs-anti}}(a,2,s+1).
$
\end{lemma}

\begin{proof}
Let
\[
    c=c_{\ref{lem:hit-vs-anti}}(a,2,s+1).
\]
If $|V(G)|\leq 1$, then $Z=V(G)$ satisfies the conclusion. Thus, assume that $|V(G)|\geq 2$, and choose a branch decomposition $\mathcal D=(T,\delta)$ of $G$ with $\simw_G(\mathcal D)\leq s$. Assign weight $\wei(v)$ to the leaf $\delta(v)$ and weight zero to every non-leaf of $T$. Let $x$ be a weighted centroid\footnote{For a tree $T$ equipped with a nonnegative vertex-weighting
$\mu$, a \defn{weighted centroid} is a vertex $x\in V(T)$ such that
$\mu(V(Q))\leq \mu(V(T))/2$ for every component $Q$ of $T-x$.
Every vertex-weighted tree has a weighted centroid.} of $T$.
If $x=\delta(v)$ is a leaf, then the total weight of all other leaves is at most $\wei(G)/2$. Hence $\{v\}$ is a balanced separator of $(G,\wei)$ and $\alpha(G[\{v\}])=1<2c$. We may therefore assume that $x$ is not a leaf. Let $T_1,\ldots,T_p$ be the components of $T-x$, where $p\in\{2,3\}$, and let
\[
    A_i=\{v\in V(G)\mid \delta(v)\in V(T_i)\}
    \qquad\text{for every }i\in[p].
\]
Then $\wei(A_i)\leq\wei(G)/2$ for every $i\in[p]$.

For every $i\in[p-1]$, the edge joining $x$ to $T_i$ displays the cut
$(A_i,\overline{A_i})$. Since
$
    \mathsf{cutsim}_G(A_i,\overline{A_i})\leq s<s+1,
$
\cref{lem:sim-cut-hitting-set} gives a set $X_i\subseteq V(G)$ that meets every edge between $A_i$ and $\overline{A_i}$ and satisfies $\alpha(G[X_i])<c$. Set
\[
    Z=\bigcup_{i=1}^{p-1}X_i.
\]
Then
\[
    \alpha(G[Z])
    \leq\sum_{i=1}^{p-1}\alpha(G[X_i])
    <(p-1)c
    \leq 2c.
\]
Moreover, no edge of $G-Z$ has endpoints in distinct sets among $A_1,\ldots,A_p$. Indeed, if $uv\in E(G-Z)$ with $u\in A_i$ and $v\in A_j$, where $i<j$, then $i\leq p-1$ and $X_i$ meets $\{u,v\}$, a contradiction. Hence every component $C$ of $G-Z$ is contained in some $A_i$, and therefore
$
    \wei(V(C))\leq\wei(A_i)\leq\frac{\wei(G)}{2}.
$
Thus, $Z$ is a balanced separator of $(G,\wei)$. This proves~\cref{lem:simw-balanced-separator}.
\end{proof}

We next relate induced bicliques in $G_H^\circ$ to rooted
$H$-obstructions. We use~\cref{thm:ramsey} by applying
it to an auxiliary edge-colored complete graph.

Let $H$ be a connected graph on $h$ vertices. For $r\in\mathbb N$,
set
\[
    q_H(r)=
    \begin{cases}
        r, & \text{if }h=1,\\
        (h+1)\Ram_{h^2}\bigl(\underbrace{2r,\ldots,2r}_{h^2\text{ times}}\bigr), & \text{if }h\geq2.
    \end{cases}
\]

\begin{lemma}\label{lem:biclique-to-rooted-H-obstruction}
Let $H$ be a connected graph and let $r\in\mathbb N$. If
$G_H^\circ$ contains an induced copy of
$K_{q_H(r),q_H(r)}$, then $G$ contains a rooted
$H$-obstruction of order $r$ as an induced subgraph.
\end{lemma}

\begin{proof}
Let $h=|V(H)|$. If $h=1$, then $G_H^\circ$ is isomorphic to $G$,
and an induced $K_{r,r}$ in $G$ is a rooted $H$-obstruction of
order $r$. We may therefore assume that $h\geq2$.
Let $(\mathcal X,\mathcal Y)$ be the bipartition of an induced copy
of $K_{q_H(r),q_H(r)}$ in $G_H^\circ$, and let
$
    N=\Ram_{h^2}\bigl(\underbrace{2r,\ldots,2r}_{h^2\text{ times}}\bigr).
$
Since $\mathcal X$ and $\mathcal Y$ are independent sets in
$G_H^\circ$, both are induced $H$-packings in $G$.

Choose a subfamily $\mathcal X_0\subseteq\mathcal X$ of size $N$
and set
$
    U=\bigcup_{X\in\mathcal X_0}X.
$
Since $|U|=hN$ and the members of $\mathcal Y$ are pairwise
disjoint, at most $hN$ members of $\mathcal Y$ intersect $U$. As
$
    |\mathcal Y|=(h+1)N,
$
there is a subfamily $\mathcal Y_0\subseteq\mathcal Y$ of size $N$
such that
\[
    U\cap\bigcup_{Y\in\mathcal Y_0}Y=\emptyset.
\]
Write
\[
    \mathcal X_0=\{X_1,\ldots,X_N\}
    \qquad\text{and}\qquad
    \mathcal Y_0=\{Y_1,\ldots,Y_N\},
\]
and fix isomorphisms
\[
    \varphi_i:H\to G[X_i]
    \qquad\text{and}\qquad
    \psi_i:H\to G[Y_i]
    \qquad\text{for every }i\in[N].
\]

We define an edge-coloring of the complete graph on vertex set
$[N]$. Let $1\leq i<j\leq N$. The vertices $X_i$ and $Y_j$ are
adjacent in $G_H^\circ$. Since $X_i\cap Y_j=\emptyset$, there is an
edge of $G$ between $X_i$ and $Y_j$. Choose one such edge and
color the edge $ij$ with the pair $(x,y)\in V(H)^2$ for which the
chosen edge is
$
    \varphi_i(x)\psi_j(y).
$
This is an edge-coloring with at most $h^2$ colors.

Since $N=\Ram_{h^2}\bigl(\underbrace{2r,\ldots,2r}_{h^2\text{ times}}\bigr)$, it follows from
\cref{thm:ramsey} that there are indices
$
    i_1<i_2<\cdots<i_{2r}
$
and vertices $u,v\in V(H)$ such that every edge of the clique on
$\{i_1,\ldots,i_{2r}\}$ has color $(u,v)$. Set
\[
    I=\{i_1,\ldots,i_r\}
    \qquad\text{and}\qquad
    J=\{i_{r+1},\ldots,i_{2r}\}.
\]
For every $i\in I$ and $j\in J$, we have $i<j$, and hence the
definition of the coloring gives
$
    \varphi_i(u)\psi_j(v)\in E(G).
$
The graph induced by
$
    \bigcup_{i\in I}X_i\ \cup\ \bigcup_{j\in J}Y_j
$
is an $(H,u,v)$-obstruction of order $r$: the sets $X_i$ are
pairwise anticomplete, the sets $Y_j$ are pairwise anticomplete,
and all the required edges between the roots are present. This proves~\cref{lem:biclique-to-rooted-H-obstruction}.
\end{proof}

\subsection{Proof of~\cref{thm:simw-rooted-H-obstruction-restate}}

\begin{proof}
Suppose first that $\mathcal G$ has bounded induced $H$-packing
treewidth, say $\treepi_H(G)\leq k$ for every $G\in\mathcal G$.
Set $r=k+1$. If some $G\in\mathcal G$ contained a rooted
$H$-obstruction $F$ of order $r$ as an induced subgraph, then
\cref{lem:rooted-H-obstruction-lower-bound} and induced-subgraph
monotonicity would give
\[
k\geq\treepi_H(G)\geq\treepi_H(F)\geq r=k+1,
\]
a contradiction. Thus, no graph in $\mathcal G$ contains a rooted
$H$-obstruction of order $r$.

Conversely, suppose that there is an integer $r\geq1$ such that no
graph in $\mathcal G$ contains a rooted $H$-obstruction of order $r$.
Since
$\mathcal G$ has bounded sim-width, there is an integer
$s\in\mathbb N\cup\{0\}$ such that
$
    \simw(G)\leq s
$
for every $G\in\mathcal G$. Set
\[
    q=q_H(r)
    \qquad\text{and}\qquad
    c=c_{\ref{lem:hit-vs-anti}}(q,2,s+1).
\]

Fix a graph $G\in\mathcal G$, and let
$
    J=G_H^\circ.
$
By~\eqref{eq:simw-induced-H-graph}, we have
$
    \simw(J)
    \leq
    \simw(G)
    \leq
    s.
$
Moreover, $J$ is $K_{q,q}$-free. Indeed, if $J$ contained an induced
copy of $K_{q,q}$, then
\cref{lem:biclique-to-rooted-H-obstruction}, together with the
definition of $q=q_H(r)$, would imply that $G$ contains a rooted
$H$-obstruction of order $r$ as an induced subgraph, contrary to the
choice of $r$.

Apply~\cref{lem:simw-balanced-separator} to $J$ with $a=q$. It follows
that, for every vertex-weighting $\we$ of $J$, the weighted graph
$(J,\we)$ has a balanced separator $Z$ satisfying
$
    \alpha(J[Z])<2c.
$
In particular,
$
    \alpha(J[Z])\leq2c.
$
Hence, by~\cref{lem:sep-vs-ta},
$
    \ta(J)\leq10c.
$
Finally, \cref{lem:blob-equivalence}, applied to the singleton family
$\{H\}$, gives
$
    \treepi_H(G)
    =
    \ta(G_H^\circ)
    =
    \ta(J)
    \leq
    10c.
$
Thus, for every $G\in\mathcal G$, we have
$
    \treepi_H(G)
    \leq
    10c_{\ref{lem:hit-vs-anti}}
    \bigl(q_H(r),2,s+1\bigr).
$
This bound depends only on $H$, $r$, and the sim-width bound $s$ for
$\mathcal G$, and not on the particular graph $G$. Therefore,
$\mathcal G$ has bounded induced $H$-packing treewidth. This completes the proof of~\cref{thm:simw-rooted-H-obstruction-restate}.
\end{proof}

\paragraph{AI disclosure.} ChatGPT was used to identify and correct omissions
and minor errors and to polish the exposition. The authors take full
responsibility for the content of the paper.
\bibliographystyle{abbrvurl}
\bibliography{ref}

\end{document}